\documentclass[11pt]{llncs}

\usepackage{amssymb,amsmath,url}
\usepackage[dvips]{graphics}
\usepackage{enumerate}

\usepackage[symbol]{footmisc}

\DeclareMathOperator{\Tr}{Tr}

\newcommand{\F}{{\mathbb F}}

\newcommand{\ie}{i.e., }

\newcommand{\ea}{{\em et al. }}

\usepackage{vmargin}
\setpapersize{A4}
\setmarginsrb{2.0cm}{2.0cm}{2.0cm}{2.0cm}{0pt}{1.0cm}{0pt}{1.0cm}

\newcommand*\samethanks[1][\value{footnote}]{\footnotemark[#1]}

\begin{document}

\title{Fibre Products of Supersingular Curves and the Enumeration of Irreducible
Polynomials with Prescribed Coefficients}

\author{Omran Ahmadi\inst{1} \and Faruk G\"olo\u{g}lu\inst{2} \and Robert Granger\thanks{Corresponding author.}\inst{3}\thanks{Research supported by the Swiss National Science Foundation via grant number 200021-156420.} 
  \and Gary McGuire\inst{4}\thanks{Research supported by Science Foundation Ireland Grant 13/IA/1914.} \and Emrah Sercan Yilmaz\inst{4}\samethanks}
  
\institute{School of Mathematics, Institute for Research in Fundamental Sciences (IPM), Tehran, Iran\\
\email{oahmadid@ipm.ir}
\and
Department of Algebra, Charles University Prague\\
\email{farukgologlu@gmail.com}
\and 
School of Computer and Communication Sciences, \'Ecole polytechnique f\'ed\'erale de Lausanne, Switzerland\\
\email{robert.granger@epfl.ch}
\and
School of Mathematics and Statistics, University College Dublin, Ireland\\
\email{gary.mcguire@ucd.ie,emrahsercanyilmaz@gmail.com}}

\maketitle

\begin{abstract}
For any positive integers $n\geq 3, r\geq 1$ we present formulae for the number of 
irreducible polynomials of degree $n$ over the finite field $\F_{2^r}$ where the coefficients of $x^{n-1}$, $x^{n-2}$ and $x^{n-3}$ are zero. 
Our proofs involve counting the number of points on certain algebraic curves over finite fields, a technique
which arose from Fourier-analysing the known formulae for the $\F_2$ base field cases, reverse-engineering an economical new proof and then
extending it. This approach gives rise to fibre products of supersingular curves and makes explicit why the formulae have period $24$ in $n$. 
\end{abstract}



\section{Introduction}\label{sec:intro}

For $p$ a prime and $r \ge 1$ let $\F_q$ denote the finite field of $q = p^r$ elements. 
The number of monic irreducible polynomials in $\F_q[x]$ of degree $n$ is
usually denoted $I_q(n)$, and the formula for $I_q(n)$ is a classical result due to Gauss~\cite[pp. 602-629]{gauss}.
A natural extension problem is to determine the 
number of monic irreducible polynomials in $\F_q[x]$ of degree $n$ for which certain coefficients are prescribed. 
An interesting set of subproblems here consists of counting the number of monic irreducible polynomials of degree $n$ for which the first $l$ coefficients 
have the prescribed values $t_1,\ldots,t_l$, while the remaining coefficients are arbitrary, \ie 
considering polynomials of the form
\[
x^n + t_1 x^{n-1} + \cdots + t_l x^{n-l} + a_{l+1} x^{n-l-1} + \cdots + a_{n-1}x + a_n,
\]
whose number we denote by $I_q(n,t_1,\ldots,t_l)$. 
While the asymptotics for such subproblems have been obtained by Cohen~\cite{Cohen}, several exact results are known.
Carlitz gave formulae for $I_q(n,t_1)$ in 1952~\cite{Carlitz}, and
Kuz'min gave formulae for $I_q(n,t_1,t_2)$ in 1990~\cite{kuzmin1,kuzmin2}. 
For the base field $\F_2$, Cattell \ea reproduced Kuz'min's results in 
1999~\cite{cattell}, while in 2001 Yucas and Mullen computed $I_2(n,t_1,t_2,t_3)$ for $n$ even~\cite{yucasmullen} and Fitzgerald and Yucas
computed $I_2(n,t_1,t_2,t_3)$ for $n$ odd~\cite{fitzyucas}. In 2013 Ri \ea gave formulae for 
$I_{2^r}(n,t_1,t_2)$ for all $r \ge 1$~\cite{RMR}.
In this paper we give a formula for $I_{2^r}(n,0,0,0)$ for all $r \ge 1$.

While our aim is to count the number of monic irreducible polynomials with certain coefficients prescribed, it turns out to be easier 
(and arguably more interesting) to count the number of elements of the relevant fields with correspondingly prescribed traces; a standard M\"obius inversion-type argument allows one to switch between the two. Indeed,~\cite{cattell,yucasmullen,fitzyucas,RMR} all take this approach.
With this in mind, for $r \ge 1$ let $q = 2^r$. For $a \in \F_{q^n}$, the characteristic polynomial of $a$ w.r.t. the extension $\F_{q^n}/\F_q$ is 
defined to be: 
\begin{equation}\label{characteristicpoly}
\prod_{i = 0}^{n-1} (x + a^{q^i}) = x^n + a_{n-1}x^{n-1} + \cdots + a_1x + a_0.
\end{equation}
The coefficient $a_{n-1}$ is known as the trace of $a$, while $a_{n-2}$ is known as the subtrace. We are concerned with the first three traces, 
which we henceforth denote by $T_1,T_2$ and $T_3$.
By~(\ref{characteristicpoly}) they are given by the following expressions:
\begin{eqnarray}
\nonumber T_1(a) &=& \sum_{i = 0}^{n-1} a^{q^i},\\
\nonumber T_2(a) &=& \sum_{0 \le i < j \le n-1} a^{q^i + q^j},\\
\nonumber T_3(a) &=& \sum_{0 \le i < j <k \le n-1} a^{q^i + q^j + q^k}.
\end{eqnarray}
For $t_1 \in \F_q$, let $F_q(n,t_1)$ be the number of elements $a \in \F_{q^n}$ for which $T_1(a) = t_1$. 
Similarly, for $t_1,t_2 \in \F_q$ let $F_q(n,t_1,t_2)$ be the number of elements $a \in \F_{q^n}$ for which
$T_1(a) = t_1$ and $T_2(a) = t_2$. Lastly, for $t_1,t_2,t_3 \in \F_q$ let $F_q(n,t_1,t_2,t_3)$ be the number 
of elements $a \in \F_{q^n}$ for which $T_1(a) = t_1, T_2(a) = t_2$ and $T_3(a) = t_3$. 

The exact counts of elements with prescribed traces found in~\cite{cattell,yucasmullen,fitzyucas,RMR} are intriguing since they depend on $n \bmod 8$ in the two trace ($l=2$) cases and on $n \bmod 24$ in the three trace ($l=3$) cases. In this work we give a new explanation for this phenomenon. It happens because the formulae are derived from the number of 
rational points on certain curves -- and the important point is that these curves are supersingular. The supersingularity implies that the 
roots of the Weil polynomials are roots of unity. The presence of these roots of unity explains the periodicity.

The algebraic curves that arise in this paper have a certain form, namely
$S(y)=T(x)$ where $T(x)=xR_1(x)+(xR_2(x))^2+\cdots +(xR_m(x))^{2^{m-1}}$
and each $S, R_i$ is a 2-linearised polynomial.
Such curves have been studied by van der Geer-van der Vlugt~\cite{Geer1995}.
Following them, we show that our curves can be viewed as fibre products of simpler curves,
and we compute the trace of Frobenius for the simpler curves.

Finding exact formulae for $F_q(n,t_1,\ldots,t_l)$ when $l > 3$ is an open and apparently much harder problem. We note that Koma has given 
approximations for these counts in the case $q = 2$ and $l = 4$, when $n \equiv 0,2 \pmod{4}$~\cite{Koma}. In forthcoming work by the third 
author, by using a similar curve-based approach it is shown that the formulae for when more than three coefficients are prescribed 
are in general not periodic, as they do not arise from supersingular curves alone. This fact explains why previous approaches to this problem -- which did not use our curve-based approach -- have failed to make any progress in the past $15$ years. Thus, in the present work we explain the periodicity of known formulae, extend this to larger base fields, and lay the foundations for further developments in this area.

The paper is organised as follows. In Section~\ref{sec:irreduciblecounts} we present the formula for $I_q(n,0,0,0)$ in
terms of $F_q(n,0,0,0)$, then in Section~\ref{backg} we provide some background on curves and their Jacobians.
In Section~\ref{sec:motivation} we motivate our approach by considering the $\F_2$ base field case with 
the first two traces specified. We then extend this and consider the same base field but with 
the first three traces specified, in Section~\ref{sec:3coeffsF2}. In Sections~\ref{sec:2r}, \ref{sec:rationalpoints} and~\ref{sec:explicit},
we compute $F_q(n,0,0,0)$ for all finite binary fields, and make some concluding remarks in Section~\ref{sec:conclusion}.



\section{Computing $I_q(n,0,0,0)$ from $F_q(n,0,0,0)$}\label{sec:irreduciblecounts}

In this section we express $I_q(n,0,0,0)$ in terms of $F_q(n,0,0,0)$, using simple extensions of results from~\cite{cattell} and~\cite{yucasmullen}. 
Recall that $I_q(n)$ is the number of irreducible polynomials of degree $n$ in $\F_q[x]$ and $I_q(n,t_1)$ is the number of irreducible polynomials 
of degree $n$ in $\F_q[x]$ with first coefficient $t_1$. Furthermore, let $\mu(\cdot)$ be the M\"obius function, and for a proposition $P$ let 
$[P]$ denote its truth value, \ie $[P] = 1$ if $P$ is true and $0$ if $P$ is false, with $1$ and $0$ interpreted as integers. The 
formula is given by the following theorem.
\begin{theorem}\label{thm:irreducibles}
Let $n \ge 3$. Then
\begin{eqnarray}
\label{thm91} I_q(n,0,0,0) &=& \frac{1}{n} \sum_{\substack{d \mid n \\ d \ \emph{odd}}} \mu(d) \big( F_q(n/d,0,0,0) - [n \ \emph{even}]q^{n/2d - 1} \big)\\
\label{thm92}             &=& \frac{1}{n} \sum_{\substack{d \mid n \\ d \ \emph{odd}}} \mu(d) F_q(n/d,0,0,0) - [n \ \emph{even}]I_{\sqrt{q}}(n,1). 
\end{eqnarray}
\end{theorem}

For $\beta \in \F_{q^n}$ let $p = \text{Min}(\beta)$ denote the minimum polynomial of $\beta$ over $\F_q$, which has degree $n/d$ for some $d \mid n$.
Note that $T_i(\beta)$ is the coefficient of $x^{n-i}$ in $p^d$~\cite[Lemma 2]{cattell}, and so abusing notation slightly we also write $T_i(\beta)$ as 
$T_i(p^d)$.
We use the following easy lemma, which was stated for the $q = 2$ case in~\cite[Prop. 1]{yucasmullen}.
\begin{lemma}
For each integer $d \ge 1$ and $p(x) \in \F_q[x]$,
\begin{enumerate}
\item $T_1(p^d) = d T_1(p)$
\item $T_2(p^d) = \binom{d}{2} T_1(p) + d T_2(p)$
\item $T_3(p^d) = \binom{d}{3} T_1(p) + d T_3(p)$
\end{enumerate}
\end{lemma}
Note that $\binom{d}{i} \bmod{2}$ for $i = 1,2,3$ depends on $d \bmod{4}$. In particular, $\binom{d}{2}$ is even if and only if 
$d \equiv 0,1 \pmod{4}$ and $\binom{d}{3}$ is even if and only if $d \equiv 0,1,2 \pmod{4}$.
Throughout we therefore write $d \equiv a$ to denote $d \equiv a \pmod{4}$.
As in~\cite{cattell} let ${\bf Irr}(n)$ denote the set of all irreducible polynomials of degree $n$ over $\F_q$, 
and let $a \cdot {\bf Irr}(n)$ denote the multiset consisting of $a$ copies of ${\bf Irr}(n)$. 
As in~\cite{cattell} and~\cite{yucasmullen} we have:
\begin{eqnarray}
\nonumber F_q(n,t_1,t_2,t_3) &=& \Big\vert \bigcup_{\beta \in \F_{q^n}, T_1(\beta) = t_1, T_2(\beta) = t_2, T_3(\beta) = t_3} Min(\beta) \Big\vert\\
\nonumber                    &=& \Big\vert \bigcup_{d \mid n} \frac{n}{d} \big\{ p \in {\bf Irr}(\frac{n}{d}):  dT_1(p) = t_1, 
                                 \binom{d}{2}T_1(p) + d T_2(p) = t_2, \binom{d}{3}T_1(p) + d T_3(p) = t_3 \big\} \Big\vert \\
\nonumber                    &=& \Big\vert \bigcup_{d \mid n, \ d \equiv 0} \frac{n}{d} \big\{ p \in {\bf Irr}(\frac{n}{d}):  
                                 0 = t_1, 0 = t_2, 0 = t_3 \big\} \Big\vert \\
\nonumber        	            &+& \Big\vert \bigcup_{d \mid n, \ d \equiv 1} \frac{n}{d} \big\{ p \in {\bf Irr}(\frac{n}{d}):  
                                 T_1(p) = t_1, T_2(p) = t_2, T_3(p) = t_3 \big\} \Big\vert \\ 
\nonumber        	            &+& \Big\vert \bigcup_{d \mid n, \ d \equiv 2} \frac{n}{d} \big\{ p \in {\bf Irr}(\frac{n}{d}):  
                                 0 = t_1, T_1(p) = t_2, 0 = t_3 \big\} \Big\vert \\
\nonumber        	            &+& \Big\vert \bigcup_{d \mid n, \ d \equiv 3} \frac{n}{d} \big\{ p \in {\bf Irr}(\frac{n}{d}):  
                                 T_1(p) = t_1, T_1(p) + T_2(p) = t_2, T_1(p) + T_3(p) = t_3 \big\} \Big\vert . 
\end{eqnarray}
For $t_1 = t_2 = t_3 = 0$ we have
\begin{equation}\label{eq:F000}
F_q(n,0,0,0) = \sum_{\substack{d \mid n \\ d \equiv 0}} \frac{n}{d} I_q(\frac{n}{d}) 
+ \sum_{\substack{d \mid n \\ d \equiv 2}} \frac{n}{d} I_q(\frac{n}{d},0)
+ \sum_{\substack{d \mid n \\ d \ \text{odd}}} \frac{n}{d} I_q(\frac{n}{d},0,0,0)       
\end{equation}
To evaluate the sum of the first two terms of~(\ref{eq:F000}) we employ the following analogue of~\cite[Lemma 5]{cattell}.
\begin{lemma}
\[
\sum_{\substack{d \mid n \\ d \equiv 0}} \frac{n}{d} I_q(\frac{n}{d}) 
+ \sum_{\substack{d \mid n \\ d \equiv 2}} \frac{n}{d} I_q(\frac{n}{d},0) = [n \ \emph{even}]\, q^{n/2 - 1}.
\]
\end{lemma}

\begin{proof}
Recall the classical result of Gauss~\cite[pp. 602-629]{gauss}, and that of Carlitz~\cite{Carlitz}, for $t_1 \ne 0$, respectively:
\begin{equation}
\label{carlitzformula} I_q(n) = \frac{1}{n} \sum_{d \mid n} \mu(d) q^{n/d}, \hspace{3mm}
I_q(n,t_1) = \frac{1}{qn} \sum_{\substack{d \mid n \\ d \ \text{odd}}} \mu(d) q^{n/d},
\end{equation}
and consequently $I_q(n,0) = I_q(n) - (q-1)I_q(n,1)$.
Let 
\[
A(n) = (q-1) \sum_{\substack{d \mid n \\ d \equiv 2}} \frac{n}{d} I_q(\frac{n}{d},1)
\] 
and let 
\[
B(n) = \sum_{\substack{d \mid n \\ d \equiv 0}} \frac{n}{d} I_q(\frac{n}{d}) 
+ \sum_{\substack{d \mid n \\ d \equiv 2}} \frac{n}{d} I_q(\frac{n}{d},0).
\]
If $n$ is odd then the sums in $B(n)$ are empty and hence the Lemma is true. Therefore let $n$ be even.
By essentially the same argument given in the proof of~\cite[Lemma 5]{cattell}, we have $A(n) = (q-1)q^{n/2-1}$.
Similarly, we have
\begin{eqnarray}
\nonumber B(n) + A(n) &=& \sum_{\substack{d \mid n \\ d \equiv 0}} \frac{n}{d} I_q(\frac{n}{d}) 
                          + \sum_{\substack{d \mid n \\ d \equiv 2}} \frac{n}{d} I_q(\frac{n}{d},0) 
                          + (q-1) \sum_{\substack{d \mid n \\ d \equiv 2}} \frac{n}{d} I_q(\frac{n}{d},1) \\
\nonumber &=&  \sum_{\substack{d \mid n \\ d \ \text{even}}} \frac{n}{d} I_q(\frac{n}{d}) 
           = \sum_{d \mid \frac{n}{2}} \frac{n}{2d} I_q(\frac{n}{2d}) = q^{n/2}.
\end{eqnarray}
Hence $B(n) = q^{n/2} - A(n) = q^{n/2} - (q-1)q^{n/2 - 1} = q^{n/2-1}$. \qed
\end{proof}
Using this result,~(\ref{eq:F000}) becomes
\[
F_q(n,0,0,0) - [n \ \text{even}]\, q^{n/2 - 1} = \sum_{\substack{d \mid n \\ d \ \text{odd}}} \frac{n}{d} I_q(\frac{n}{d},0,0,0).
\]
Applying~\cite[Corollary 1]{cattell} with $\alpha(n) = n I_q(n,0,0,0)$ finally proves expression~(\ref{thm91}) of Theorem~\ref{thm:irreducibles}, and
applying Carlitz's result from~(\ref{carlitzformula}) gives~(\ref{thm92}).


\section{Background on curves and Jacobians}\label{backg}


Let $p$ be a prime number and let $q=p^r$.
Let $C$ be a curve of genus $g$ defined over $\mathbb{F}_q$.
All curves in this paper are projective plane curves, lying in $\mathbb{P}^2(k)$
where $k$ is an algebraic closure of $\mathbb{F}_q$.
As is common, we will work with affine versions of our curves and include the point at infinity in our point counts. 
Note that we do not include the point at infinity when counting the relevant $F_q(n,t_1,t_2,t_3)$.

We denote by $J_C$ the Jacobian 
of $C$, an abelian variety which as a group is isomorphic to
the divisor class group $\textrm{Pic}^0 (C)$
(the degree 0 divisors modulo the principal divisors).

The map $(x,y)\mapsto (x^q,y^q)$ on $C$  (defined over an
algebraic closure of $\mathbb{F}_q$) induces an endomorphism on $J_C$
called the Frobenius endomorphism.
The characteristic polynomial of the Frobenius endomorphism plays a key role.
The roots of this characteristic polynomial are called
Frobenius eigenvalues.

\subsection{Abelian varieties}
Suppose that $A$ and $B$ are Abelian varieties over the same field $K$.
A homomorphism $f : A\longrightarrow B$ is 
 a morphism that is also a group homomorphism. A homomorphism
is an isogeny over $K$ if $f$ is surjective and defined over $K$
and $\dim(A) = \dim(B)$.  If $A$ is an Abelian variety over $K$, 
$A$ is called simple (over $K$) if it
is not isogenous over $K$ to a product of lower-dimensional Abelian varieties.

Because $J_C$ is an abelian variety
we continue with some facts from the theory of abelian varieties.

Let $A$ be any abelian variety of dimension $g$ over $\mathbb{F}_q$. 
It can be shown that $A$ has a Frobenius endomorphism, and that
the characteristic polynomial $P_A(t)\in \mathbb{Z}[t]$ 
of the Frobenius endomorphism has the form
\begin{displaymath}
P_A(t) = t^{2g} + a_{1}t^{2g-1} + \cdots+ a_{g-1}t^{g+1}+
a_gt^g +qa_{g-1}t^{g-1}+ \cdots + a_1 q^{g-1} t+q^g.
\end{displaymath}
The Weil polynomial of an abelian variety over a finite field $\F_q$ is the
characteristic polynomial of its Frobenius endomorphism. 
The Weil polynomial of a curve over $\F_q$ is the Weil polynomial of its Jacobian.

The isogeny classes of abelian varieties
are completely classified by their Weil polynomials,
as the following theorem of Tate shows.

\begin{theorem}(Tate) \label{Tate} 
Let $A$ and $B$ be abelian varieties defined over $\mathbb{F}_q$. Then an
abelian variety $A$ is $\mathbb{F}_q$-isogenous to an abelian subvariety of
$B$ if and only if $P_A(t)$ divides $P_B(t)$ over $\mathbb{Q}[t]$. In
particular, $P_A(t) = P_B(t)$ if and only if $A$ and $B$ are
$\mathbb{F}_q$-isogenous.
\end{theorem}

\subsection{L-polynomials}

For a curve $C$ over $\F_q$ we
define for any $n\geq 1$ 
the number $N_n(C)$ of $\F_{q^n}$ rational points of $C$, 
and we encode these numbers
into the zeta function of $C$
\[
Z_C(t ) = exp \biggl( \sum_{n\geq 1} N_n (C)  \frac{t^n}{n} \biggr).
\]
It can be shown that the zeta function is a rational function of $t$
with denominator $(1-t)(1-qt)$.
The numerator of the zeta function of a curve $C$
is called the L-polynomial of $C$.
It can be shown that
the L-polynomial of $C$ is the reverse polynomial of the Weil polynomial of $J_C$, 
which we denote $P_C(t)$, i.e.,
\[
L_C (t)=t^{2g} P_{C} (1/t).
\]
Let $P_{C} (t)=\prod_{i=1}^{2g}(t-\eta_i)$, so
$L_{C} (t)=\prod_{i=1}^{2g}(1-\eta_i t)$.
 It can be shown that if $L_n(t)$ denotes the L-polynomial of 
 C over the extension $\F_{q^n}$,
then $L_n(t)=\prod_{i=1}^{2g} (1-\eta_i^n t)$.
We also have that
$N_n:=|C(\F_{q^n})|=q^n+1-\sum_{i=1}^{2g} \eta_i^n$ for all $n\geq 1$,
and the Riemann Hypothesis $|\eta_i|=\sqrt{q}$,
which together yield 
$|N_n -(q^n+1)| \leq 2g\sqrt{q^n}$ (the Hasse-Weil bound).

In this paper, all L-polynomials are for curves defined over $\F_2$ unless otherwise stated.

\subsection{Supersingularity}

An elliptic curve $E$ over $\F_q$  is called supersingular if its Weil polynomial
$t^2+at+q$ has the property that $a$ is divisible by $p$.
Equivalently, $E$ is supersingular if
$E(\overline{\mathbb{F}_q})$ has no points of order $p$.
 A curve $C$ of genus $g$ over a finite field $\mathbb{F}_q$ is called supersingular if its
Jacobian is a supersingular abelian variety, i.e.,
its Jacobian is isogenous (over $\overline{\mathbb{F}_q}$) to $E^g$, where $E$ is
a supersingular elliptic curve.
  
For a supersingular curve $C$ defined over a finite field $\mathbb{F}_q$,
it was shown by Oort that $\eta_i = \sqrt{q}\cdot \zeta_i$ 
where $\zeta_i$ is a (not necessarily $i$-th) root of unity, for all $i$.

\subsection{Maximal and minimal curves}

	Let $g$ be genus of the curve $C$ defined over $\mathbb F_{2^r}$ and let $\eta_1,\cdots,\eta_{2g}$ be the roots of reciprocal of the $L$- polynomial of $C$ over $\mathbb F_{2^r}$. Then the number of rational points of $C$ on $\mathbb F_{2^{rn}}$ is $$\#C(\mathbb F_{2^{rn}})=(2^{rn}+1)- \sum\limits_{i=1}^{2g}\eta_i^n.$$ 
	As stated earlier $|\eta_i|=\sqrt{2^r}$ for all $i=1,\cdots,2g$. We call $C(\mathbb F_{2^{rn}})$ maximal if $\eta_i^n=-\sqrt{2^{rn}}$ for all $i=1,\cdots,2g$, so the Hasse-Weil
	upper bound is met. 
	We call $C(\mathbb F_{2^{rn}})$ minimal if $\eta_i^n=\sqrt{2^{rn}}$ for all $i=1,\cdots,2g$.

A supersingular curve defined over $\F_q$ becomes maximal over some
finite extension of $\F_q$, because $\eta_i = \sqrt{q}\cdot \zeta_i$ 
where $\zeta_i$ is a (not necessarily $i$-th) root of unity, for all $i$.

	\begin{proposition} \label{minimal-prop}
		If $C(\mathbb F_{2^{rn}})$ is maximal or minimal for some $n \geq 1$ then  $C(\mathbb F_{2^{r(2n)}})$ is minimal. 
	\end{proposition}
	\begin{proof}
		Since $\eta_i^n=\pm \sqrt{2^{rn}}$ for all $i=1,\cdots,2g$, we have  $\eta_i^{2n}=(\pm \sqrt{2^{rn}})^2=\sqrt{2^{r(2n)}}$ for all $i=1,\cdots,2g$. The proof follows by definition of a minimal curve. \qed
	\end{proof}
	
	We will need the following simple observations later. 

	\begin{lemma}\label{lemma-unity}
		The real numbers $1$ and $\sqrt 2$ cannot be written as a $\mathbb Q$-linear combination of primitive $16$-th roots of unity. Moreover,  $1$ cannot be written as a $\mathbb Q$-linear combination of primitive $8$-th roots of unity.
	\end{lemma}
	\begin{proof}
		Let $\eta_{2^n}=\exp(\pi/2^{n-1}) \in \mathbb C$ be a primitive $2^n$-th root of unity. Then $\eta_{2^n}$ is a root of $X^{2^{n-1}}+1$ which is the $2^n$-th cyclotomic polynomial and therefore irreducible over $\mathbb Q$. Let $K_n := \mathbb Q(\eta_{2^n})$. Then $1,\eta_{2^n},\cdots,\eta_{2^n}^{2^{n-1}-1}$ are linearly independent over $\mathbb Q$. Note that $\eta_{2^n}^i$ is a primitive $2^n$-th root of unity if and only if $i$ is odd.\\
		Since $1=1$ and $\sqrt{2}=\eta_{16}^2-\eta_{16}^6$,  $1$ and $\sqrt 2$ cannot be written as $\mathbb Q$-linear combination of primitive $16$-th roots of unity. Since $1=1$,  $1$ cannot be written as $\mathbb Q$-linear combination of primitive $8$-th roots of unity. \qed
	\end{proof}

\subsection{Quadratic forms}\label{QF}

We now recall the basic theory of quadratic forms over $\mathbb{F}_{2}$,
which we will use later.

Let $K=\mathbb{F}_{2^m}$, and 
let $Q:K\longrightarrow \mathbb{F}_{2}$ be a quadratic form.
The polarization of $Q$ is the symplectic bilinear form $B$ defined by
\[
B(x,y)=Q(x+y)-Q(x)-Q(y).
\]
By definition the radical of $B$ (denoted $W$) is 
\[
W =\{ x\in K : B(x,y)=0 \text{  for all $y\in K$}\}.
\]
The rank of $B$ is defined to be $m-\dim(W)$, and a well known theorem
states that the rank must be even.

Next let $Q|_W$ denote the restriction of $Q$ to $ W$, and let 
\[
W_0=\{ x\in W : Q(x)=0\}
\]
(sometimes $W_0$ is called the singular radical of $Q$).
Note that $Q|_W$ is a linear map $W\longrightarrow \mathbb{F}_2$
with kernel $W_0$.  Therefore
\begin{equation*}
\dim W_0 =
\begin{cases}
\dim(W) -1 & \text{if $Q|_W$ is onto} \\ 
\dim(W) & \text{if $Q|_W =0$ (i.e. $W=W_0$)}.%
\end{cases}%
\end{equation*}
The rank of $Q$ is defined to be $m-\dim(W_0)$.
The following theorem is well known, see Chapter $6$ of \cite{lidl}  for example.
\bigskip

\begin{proposition}\label{counts}
	Continue the above notation.
	Let $N=|\{x\in K : Q(x)=0\}|$, and let $w=\dim(W)$.
	
	If $Q$ has odd rank then $N=2^{m-1}$.
	In this case, $\sum_{x\in K} (-1)^{Q(x)} =0$.

	If $Q$ has even rank then $N=2^{m-1}\pm 2^{(m-2+w)/2}$.
	\end{proposition}


\section{Motivating example: $F_2(n,t_1,t_2)$}\label{sec:motivation}

Since precisely half of the elements of $\F_{2^n}$ have trace zero and half of the elements have trace one, we have $F_2(n,0) =
F_2(n,1) = 2^{n-1}$ for all $n \ge 1$. Hence the four $F_2(n,t_1,t_2)$ are the
first interesting cases. 

Cattell \ea \ have expressed $F_2(n,t_1,t_2)$ as sums of binomial coefficients~\cite[Thm. 4]{cattell}. 
For $n \ge 2$, we instead write $F_2(n,t_1,t_2) = 2^{n-2} + f_2(n,t_1,t_2)$ and deduce
the formulae for $f_2(n,t_1,t_2)$ from~\cite[Thm. 3.4]{fitzyucas} for odd $n$, and~\cite[Thm. 3]{yucasmullen} for even $n$,
which are presented in Table~\ref{table:2coeffs}, in which the rightmost four columns cover the four
$(t_1,t_2)$ pairs.
\begin{table}[t]
\caption{$f_2(n,t_1,t_2)$}
\begin{center}\label{table:2coeffs}
\begin{tabular}{c|cccc}
\hline
$n \pmod{8}$ & $(0,0)$ & $(0,1)$ & $(1,0)$ & $(1,1)$ \\
\hline
$0$ & $-2^{n/2 - 1}$     & $2^{n/2 - 1}$  & $0$   & $0$ \\
$1$ & $2^{n/2 - 3/2}$  & $-2^{n/2 - 3/2}$ & $2^{n/2 - 3/2}$ & $-2^{n/2 - 3/2}$\\
$2$ & $0$               & $0$  &  $-2^{n/2-1}$ & $2^{n/2-1}$ \\
$3$ & $-2^{n/2 - 3/2}$  & $2^{n/2 - 3/2}$  & $2^{n/2 - 3/2}$  & $-2^{n/2 - 3/2}$ \\
$4$ & $2^{n/2 - 1}$      & $-2^{n/2 - 1}$  & $0$ & $0$\\
$5$ & $-2^{n/2-3/2}$    & $2^{n/2-3/2}$ & $-2^{n/2-3/2}$ & $2^{n/2-3/2}$ \\
$6$ & $0$               &  $0$  & $2^{n/2-1}$ & $-2^{n/2-1}$ \\
$7$ & $2^{n/2 - 3/2}$  & $-2^{n/2 - 3/2}$  &  $-2^{n/2 - 3/2}$ & $2^{n/2 - 3/2}$   \\
\hline
\end{tabular}
\end{center}
\end{table}
As one can see, for each $(t_1,t_2)$ the expression for $f_2(n,t_1,t_2)$ depends on $n \bmod 8$.
This observation naturally gives rise to two related questions: why are the
formulae periodic, and why is the period $8$? 

To answer these questions we Fourier-analyse the formulae, using the complex $8$-th roots of unity. 
In particular, observe that dividing each $f_2(n,t_1,t_2)$ by $2^{n/2}$ gives a constant. We can thus 
normalise by this factor and expand each as follows:
\begin{equation}\label{fourier}
f_2(n,t_1,t_2) = \sum_{k=0}^{7} g(t_1,t_2)_k \, (\sqrt{2}\omega_{8}^{k})^n = 2^{n/2}
\, \sum_{k=0}^{7} g(t_1,t_2)_k \, \omega_{8}^{kn},
\end{equation}
where the $g(t_1,t_2)_k$ are unknown constants and $\omega_8 =  e^{i \pi/4} = (1 + i)/\sqrt{2}$.  
For each $(t_1,t_2)$, let $\overline{v}_{t_1,t_2}$ be the length-$8$ vector
$f_2(n,t_1,t_2)/2^{n/2}$, \ie
\begin{eqnarray}
\nonumber \overline{v}_{0,0} &=& (-2^{-1},2^{-3/2},0,-2^{-3/2},2^{-1},-2^{-3/2},0,2^{-3/2}),\\
\nonumber \overline{v}_{0,1} &=& (2^{-1},-2^{-3/2},0,2^{-3/2},-2^{-1},2^{-3/2},0,-2^{-3/2}),\\
\nonumber \overline{v}_{1,0} &=& (0,2^{-3/2},-2^{-1},2^{-3/2},0,-2^{-3/2},2^{-1},-2^{-3/2}),\\
\nonumber \overline{v}_{1,1} &=& (0,-2^{-3/2},2^{-1},-2^{-3/2},0,2^{-3/2},-2^{-1},2^{-3/2}).
\end{eqnarray}
Let $M$ be the discrete Fourier transform matrix for $\omega_8$ with entries $M_{j,k} =
\omega_{8}^{(j-1)\cdot(k-1)}$. Then $\overline{g}(t_1,t_2) = [g(t_1,t_2)_0,\ldots,g(t_1,t_2)_{7}]$
satisfies:
\[
M \cdot \overline{g}(t_1,t_2)^{T} = \overline{v}_{t_1,t_2}^T. 
\]
Since $M_{j,k}^{-1} = \frac{1}{8} \cdot \omega_{8}^{-(j-1)\cdot(k-1)}$, we easily compute
\[
\overline{g}(t_1,t_2)^T = M^{-1} \cdot \overline{v}_{t_1,t_2}^T.
\]
Then substituting each $\overline{g}(t_1,t_2)$ into the Fourier expansion~(\ref{fourier}), we obtain
\begin{proposition}
For $n \ge 2$ we have
\begin{eqnarray}
\nonumber f_2(n,0,0) &=& -\frac{1}{4}((\sqrt{2}\omega_{8}^3)^n +
(\sqrt{2}\omega_{8}^5)^n)  = -\frac{1}{4}((-1+i)^n + (-1-i)^n),\\
\nonumber f_2(n,0,1) &=& \frac{1}{4}((\sqrt{2}\omega_{8}^3)^n +
(\sqrt{2}\omega_{8}^5)^n) = \frac{1}{4}((-1+i)^n + (-1-i)^n),\\
\nonumber f_2(n,1,0) &=& -\frac{i}{4}((\sqrt{2}\omega_{8}^3)^{n} -
(\sqrt{2}\omega_{8}^5)^{n}) = -\frac{i}{4}((-1+i)^n - (-1-i)^n),\\
\nonumber f_2(n,1,1) &=& \frac{i}{4}((\sqrt{2}\omega_{8}^3)^{n} -
(\sqrt{2}\omega_{8}^5)^{n}) = \frac{i}{4}((-1+i)^n - (-1-i)^n).
\end{eqnarray}
\end{proposition}

By a theorem due to G\"olo\u{g}lu, McGuire and Moloney, the expressions for
$F_2(n,0,0)$, $F_2(n,1,1)$, $F_2(n,0,1)$ and $F_2(n,1,0)$ count the number of
elements $a \in \F_{2^n}$ for which the corresponding binary Kloosterman sum is
congruent to $0,4,8$ and $12 \pmod{16}$, respectively~\cite[Thm. 1.2]{faruk1}.  
Furthermore, Lison\v{e}k and Moisio have explicitly computed the number of elements
$a \in \F_{2^n}$ for which the Kloosterman sum is divisible by $16$~\cite[Thm 3.6]{lisonekmoisio}, 
relating it to the number of $\F_{2^n}$-rational points on the supersingular elliptic curve
\[
E_1/\F_2 : y^2 + y = x^3 + x.
\]
By G\"olo\u{g}lu \ea's result, this number is therefore related to $F_2(n,0,0)$.

We will make repeated use of the following easy lemma, which is an immediate generalisation of~\cite[Lemma 1.1]{fitzyucas},
which was proven for the base field $\F_2$ case.
\begin{lemma}\label{lem:T2T3}
For all $\alpha, \beta \in \F_{q^n}$ we have:
\begin{itemize}
\item[(i)] $T_2(\alpha + \beta) + T_2(\alpha) + T_2(\beta) = T_1(\alpha)T_1(\beta)
  + T_1(\alpha\beta)$
\item[(ii)] $T_3(\alpha + \beta) + T_3(\alpha) + T_3(\beta) =
  T_2(\alpha)T_1(\beta) + T_1(\alpha)T_2(\beta) 
+ T_1(\alpha^2\beta + \alpha\beta^2) + T_1(\alpha\beta)T_1(\alpha+\beta)$.
\end{itemize}
\end{lemma}

We can now reprove the connection between $F_2(n,0,0)$ and $\#E_1(\F_{2^n})$ directly as follows. 
If $T_1(a) = 0$ then $a = x^2 + x$ for two $x \in \F_{2^n}$.
By Lemma~\ref{lem:T2T3}(i), we have $T_2(x^2 + x) = T_1(x^3 + x)$, and so for $n \ge 2$ we have:
\begin{eqnarray}
\nonumber F_2(n,0,0) &=& \frac{1}{2} \, \#\{ x \in \F_{2^n} \mid T_1(x^3 + x) = 0\}\\
\nonumber            &=& \frac{1}{4} \, \#\{ (x,y) \in \F_{2^n} \times
\F_{2^n} \mid y^2 + y = x^3 + x\}\\
\nonumber           &=& \frac{1}{4} \, (\#E_1(\F_{2^n}) - 1)\\
\nonumber           &=& 2^{n-2} - \frac{(-1+i)^n + (-1 - i)^n}{4},
\end{eqnarray}
where in the final line we have used the fact that the characteristic polynomial of Frobenius of $E_1$ 
is $P_{E_1}(X) = X^2 + 2X + 2 = (X - \sqrt{2}\omega_{8}^3)(X - \sqrt{2}\omega_{8}^5)$, since the trace of $E_1$ is $-2$.
As $F_2(n,0,0) + F_2(n,0,1) = F_2(n,0) = 2^{n-1}$, we immediately have
$f_2(n,0,1) = -f_2(n,0,0)$. The periodicity for the trace zero cases thus arises from the connection
with the supsersingular elliptic curve $E_1$, whose Weil $q$-numbers are $\sqrt{2}\omega_{8}^{\pm 3}$.
For the trace one cases, note that the form of $f_2(n,1,0)$ and $f_2(n,1,1)$ implies that they do not arise from the
zeta function of an abelian variety, but are related to
$f_2(n,0,0)$ and $f_2(n,0,1)$, as derived in Lemma 2.6 of~\cite{fitzyucas} for odd $n$,
and~\cite{yucasmullen} (deduced from~\cite{cattell}) for even $n$. Since our focus is on the
application of supersingular curves to the counting problem, we omit these cases for now, but indicate how supersingular
curves may also be used for the trace one cases, for odd $n$, in the next section.


\section{Recomputing $F_2(n,0,t_2,t_3)$}\label{sec:3coeffsF2}

We now apply the approach of the previous section to compute $F_2(n,0,t_2,t_3)$.
For the general case where $t_1$ may be 0 or $1$, there are two complementary results
already known. 
Fitzgerald and Yucas gave formulae for $F_2(n,t_1,t_2,t_3)$ when $n$ is odd~\cite{fitzyucas}, while Yucas and Mullen gave
formulae for $F_2(n,t_1,t_2,t_3)$ when $n$ is even~\cite{yucasmullen}. As in Section~\ref{sec:motivation} we Fourier-analyse 
these formulae in order to determine which supersingular curves are relevant and then prove the formulae directly.

As before, for $n \ge 3$ we write $F_2(n,0,t_2,t_3) = 2^{n-3} + f_2(n,0,t_2,t_3)$. Table~\ref{table:3coeffs} gives the 
set of formulae for $f_2(n,0,t_1,t_2)$, in which the rightmost four columns cover the possible $(0,t_2,t_3)$ triples.
\begin{table}[t]
\caption{$f_2(n,0,t_2,t_3)$}
\begin{center}\label{table:3coeffs}
\begin{tabular}{c|cccc}
\hline
$n \pmod{24}$ & $(0,0,0)$ & $(0,0,1)$ & $(0,1,0)$ & $(0,1,1)$\\
\hline
$0$ & $-5 \cdot 2^{n/2 - 2}$     & $3 \cdot 2^{n/2-2}$   &  $2^{n/2-2}$       &   $2^{n/2-2}$ \\
$1$ & $3 \cdot 2^{n/2 - 5/2}$  & $-2^{n/2-5/2}$       &   $-2^{n/2-5/2}$ &  $-2^{n/2-5/2}$  \\
$2$ & $2^{n/2-2}$                & $-2^{n/2-2}$          &  $2^{n/2-2}$      &    $-2^{n/2-2}$ \\
$3$ & $0$                       & $-2^{n/2-3/2}$       & $-2^{n/2-3/2}$  &  $2^{n/2 - 1/2}$ \\
$4$ & $0$                       & $2^{n/2-1}$           &  $0$             &  $-2^{n/2-1}$   \\
$5$ & $-3 \cdot 2^{n/2-5/2}$   & $2^{n/2-5/2}$        & $2^{n/2-5/2}$    &   $2^{n/2-5/2}$ \\
$6$ & $2^{n/2-2}$                & $-2^{n/2-2}$          &   $2^{n/2-2}$     &  $-2^{n/2-2}$ \\
$7$ & $3 \cdot 2^{n/2 - 5/2}$   & $-2^{n/2-5/2}$      & $-2^{n/2-5/2}$   &   $-2^{n/2-5/2}$ \\
$8$ & $-2^{n/2 - 1}$              & $0$                 &  $-2^{n/2-1}$     & $2^{n/2}$ \\
$9$ & $0$                        & $2^{n/2-3/2}$      & $2^{n/2-3/2}$    &  $-2^{n/2 - 1/2}$  \\
$10$ & $2^{n/2-2}$                & $-2^{n/2-2}$         &  $2^{n/2-2}$      &  $-2^{n/2-2}$    \\
$11$ & $-3 \cdot 2^{n/2-5/2}$    & $2^{n/2-5/2}$     &  $2^{n/2-5/2}$    &  $2^{n/2-5/2}$  \\
$12$ & $3 \cdot 2^{n/2 - 2}$      & $-2^{n/2-2}$         &  $-3 \cdot 2^{n/2-2}$& $2^{n/2-2}$  \\
$13$ & $-3 \cdot 2^{n/2-5/2}$   & $2^{n/2-5/2}$      & $2^{n/2-5/2}$    &   $2^{n/2-5/2}$ \\
$14$ & $2^{n/2-2}$               & $-2^{n/2-2}$         &  $2^{n/2-2}$      & $-2^{n/2-2}$  \\
$15$ & $0$                       & $2^{n/2-3/2}$      &  $2^{n/2-3/2}$  &   $-2^{n/2 - 1/2}$  \\
$16$ & $-2^{n/2 - 1}$            & $0$                  &  $-2^{n/2-1}$    &    $2^{n/2}$   \\
$17$ & $3 \cdot 2^{n/2 - 5/2}$   & $-2^{n/2-5/2}$     & $-2^{n/2-5/2}$  &  $-2^{n/2-5/2}$  \\
$18$ & $2^{n/2-2}$                 & $-2^{n/2-2}$       &  $2^{n/2-2}$      & $-2^{n/2-2}$    \\
$19$ & $-3 \cdot 2^{n/2 - 5/2}$   & $2^{n/2-5/2}$     & $2^{n/2-5/2}$   &  $2^{n/2-5/2}$  \\
$20$ & $0$                        & $2^{n/2-1}$        & $0$              &   $-2^{n/2-1}$\\
$21$ & $0$                       & $-2^{n/2-3/2}$     & $-2^{n/2-3/2}$  &   $2^{n/2 - 1/2}$   \\
$22$ & $2^{n/2-2}$                 & $-2^{n/2-2}$       &  $2^{n/2-2}$      &  $-2^{n/2-2}$ \\
$23$ & $3 \cdot 2^{n/2 - 5/2}$   & $-2^{n/2-5/2}$    & $-2^{n/2-5/2}$   &  $-2^{n/2-5/2}$ \\
\hline
\end{tabular}
\end{center}
\end{table}
As in Section~\ref{sec:motivation}, since the set of formulae have period $24$, one can express
each in terms of the complex $24$-th roots of unity. We thus expand each as follows:
\begin{equation}\label{fourier2}
f_2(n,0,t_2,t_3) = \sum_{k=0}^{23} g(0,t_2,t_3)_k \, (\sqrt{2} \omega_{24}^{k})^n = 
2^{n/2} \, \sum_{k=0}^{23} g_k \, \omega_{24}^{kn},
\end{equation}
with $\omega_{24} = e^{i \pi /12}  = ((1+\sqrt{3}) + (-1+\sqrt{3})i)/2\sqrt{2}$.
For each $(t_2,t_3)$, let $\overline{v}_{0,t_2,t_3}$ be the vector of constants
$f_2(n,0,t_2,t_3)/2^{n/2}$. 
Let $M$ be the discrete Fourier transform matrix with $M_{j,k} =
\omega_{24}^{(j-1)\cdot(k-1)}$. Then $\overline{g}(0,t_2,t_3) = [g(0,t_2,t_3)_k]_{0 \le k \le 23}$
satisfies:
\[
M \cdot \overline{g}(0,t_2,t_3)^{T} = \overline{v}_{0,t_2,t_3}^T. 
\]
Since $M_{j,k}^{-1} = \frac{1}{24} \omega_{24}^{-(j-1)\cdot(k-1)}$ we easily compute
\[
\overline{g}(0,t_2,t_3)^T = M^{-1} \cdot \overline{v}_{0,t_2,t_3}^{T}.
\]
Then substituting each $\overline{g}(t_1,t_2)$ into the Fourier expansion~(\ref{fourier2}), we obtain

\begin{proposition}\label{prop:F2formulae}
For $n \ge 3$ we have
\begin{eqnarray}
\nonumber f_2(n,0,0,0) &=& -\frac{1}{4}((\sqrt{2}\omega_{8}^3)^n + (\sqrt{2}\omega_{8}^5)^n)
-\frac{1}{8}((\sqrt{2}i)^n + (-\sqrt{2}i)^n)\\
\nonumber & & -\frac{1}{8}((\sqrt{2}\omega_{24}^{5})^n + (\sqrt{2}\omega_{24}^{11})^n
+ (\sqrt{2}\omega_{24}^{13})^n + (\sqrt{2}\omega_{24}^{19})^n)\\
\nonumber f_2(n,0,0,1) &=& \frac{1}{8}((\sqrt{2}i)^n + (-\sqrt{2}i)^n)\\
\nonumber & & +\frac{1}{8}((\sqrt{2}\omega_{24}^{5})^n + (\sqrt{2}\omega_{24}^{11})^n
+ (\sqrt{2}\omega_{24}^{13})^n + (\sqrt{2}\omega_{24}^{19})^n)\\
\nonumber f_2(n,0,1,0) &=& -\frac{1}{8}((\sqrt{2}i)^n + (-\sqrt{2}i)^n)\\
\nonumber & & +\frac{1}{8}((\sqrt{2}\omega_{24}^{5})^n + (\sqrt{2}\omega_{24}^{11})^n
+ (\sqrt{2}\omega_{24}^{13})^n + (\sqrt{2}\omega_{24}^{19})^n)\\
\nonumber f_2(n,0,1,1) &=& \frac{1}{4}((\sqrt{2}\omega_{8}^3)^n + (\sqrt{2}\omega_{8}^5)^n)
+ \frac{1}{8}((\sqrt{2}i)^n + (-\sqrt{2}i)^n)\\
\nonumber & & -\frac{1}{8}((\sqrt{2}\omega_{24}^{5})^n + (\sqrt{2}\omega_{24}^{11})^n + (\sqrt{2}\omega_{24}^{13})^n + (\sqrt{2}\omega_{24}^{19})^n).
\end{eqnarray}
\end{proposition}

As before the $(\sqrt{2}\omega_{8}^3)^n + (\sqrt{2}\omega_{8}^5)^n$ term
arises from the roots of the characteristic polynomial of Frobenius $P_{E_1}(X) = X^2 + 2X + 2$. The term
$(\sqrt{2}i)^n + (-\sqrt{2}i)^n$ arises from the roots of $P_{E_2}(X) = X^2 + 2$ corresponding to the elliptic curve
\[
E_2/\F_2: y^2 + y = x^3 + 1.
\] 
Finally the term $(\sqrt{2}\omega_{24}^{5})^n + (\sqrt{2}\omega_{24}^{11})^n + (\sqrt{2}\omega_{24}^{13})^n + (\sqrt{2}\omega_{24}^{19})^n$
arises from the roots of $P_{H_1}(X) = X^4 + 2X^3 + 2X^2+ 4X + 4$ corresponding to the genus $2$ supersingular curve (see~\cite[\S4]{xing})
\[
H_1/\F_2: y^2 + y = x^5 + x^3.
\]

In order to prove Proposition~\ref{prop:F2formulae} starting from the relevant supersingular curves, we use the following easy lemma
due to Fitzgerald. For a function $q_i: \F_{2^n} \rightarrow \F_2$ let $Z(q_i)$ denote the number of zeros of $q_i$, and for $w_i \in \F_2$ 
let $N(w_1,w_2)$ be the number of $x \in \F_{2^n}$ such that $q_i(x) = w_i$, for $i=1,2$.
\begin{lemma}~\cite[Prop. 5.3]{fitz}\label{lem:fitz}
Let $q_1,q_2: \F_{2^n} \rightarrow \F_2$ be functions. Then 
\[
N(0,0) = \frac{1}{2} \big( Z(q_1) + Z(q_2) + Z(q_1+q_2) - 2^n \big).
\]
\end{lemma}
Note that Fitzgerald stipulates that $q_1,q_2$ be quadratic forms, since his application requires only this; 
the proof however only requires that they be functions. We prove a generalisation of 
Lemma~\ref{lem:fitz} in Section~\ref{sec:2r}. We now prove Proposition~\ref{prop:F2formulae}.

\begin{proof}
By Lemma~\ref{lem:T2T3}(i) we have $T_2(x^2 + x) = T_1(x^3 + x)$ and by Lemma~\ref{lem:T2T3}(ii), we
have $T_3(x^2 + x) = T_1(x^5 + x)$. Therefore, let 
\begin{eqnarray}
\nonumber q_1 &:& \F_{2^n} \rightarrow \F_2: x \mapsto T_1(x^3 + x),\\
\nonumber q_2 &:& \F_{2^n} \rightarrow \F_2: x \mapsto T_1(x^5 + x).
\end{eqnarray}
For $n \ge 3$ we have:
\begin{eqnarray}
\nonumber F_2(n,0,0,0) &=& \frac{1}{2} \, \#\{ x \in \F_{2^n} \mid T_1(x^3 + x) = T_1(x^5 + x) = 0\}\\
\nonumber                        &=& \frac{1}{2} \, \#\{ x \in \F_{2^n} \mid q_1(x) = q_2(x) = 0\}\\
\nonumber  &=& \frac{1}{4} \, (Z(q_1) + Z(q_2) + Z(q_1 + q_2) - 2^n), 
\end{eqnarray}
by Lemma~\ref{lem:fitz}. Treating these zero-counts in turn, we have
\begin{eqnarray}
\nonumber Z(q_1) &=& \frac{1}{2} \, \#\{ (x,y) \in \F_{2^n} \times \F_{2^n} \mid y^2 + y = x^3 + x\}\\
\nonumber           &=& \frac{1}{2} \, (\#E_1(\F_{2^n}) - 1)\\
\nonumber           &=& \frac{1}{2} \, (2^n - (\sqrt{2}\omega_{8}^3)^n - (\sqrt{2}\omega_{8}^5)^n).
\end{eqnarray}
Secondly, we have
\begin{eqnarray}
\nonumber  Z(q_2) &=& \frac{1}{2} \, \#\{ (x,y) \in \F_{2^n} \times \F_{2^n} \mid y^2 + y = x^5 + x\}\\
\nonumber           &=& \frac{1}{2} \, (\#H_2(\F_{2^n}) - 1)\\
\nonumber           &=& \frac{1}{2} \, (2^n - (\sqrt{2}i)^n - (-\sqrt{2}i)^n - (\sqrt{2}\omega_{8}^3)^n - (\sqrt{2}\omega_{8}^5)^n ),
\end{eqnarray}
since the Jacobian of $H_2/\F_2:y^2 + y = x^5 + x$ is $\F_2$-isogenous to the product of the two supersingular elliptic curves $E_1$ 
and $E_2$~\cite[\S4]{xing}.
Finally, we have 
\begin{eqnarray}
\nonumber Z(q_1 + q_2) &=& \frac{1}{2} \, \#\{ (x,y) \in \F_{2^n} \times \F_{2^n} \mid y^2 + y = x^5 + x^3\}\\
\nonumber           &=& \frac{1}{2} \, (\#H_1(\F_{2^n}) - 1)\\
\nonumber           &=& \frac{1}{2} \, (2^n - (\sqrt{2}\omega_{24}^{5})^n -
(\sqrt{2}\omega_{24}^{11})^n - (\sqrt{2}\omega_{24}^{13})^n - (\sqrt{2}\omega_{24}^{19})^n ),
\end{eqnarray}
Combining these as per Lemma~\ref{lem:fitz} gives $2^{n-3} + f_2(n,0,0,0)$ as required. 
Furthermore, since $F_2(n,0,0,0) + F_2(n,0,0,1) = F_2(n,0,0)$, the formula for $f_2(n,0,0,1)$ follows. 

For the remaining two, let $F_2(n,0,*,0)$ be the count of those $a \in \F_{2^n}$ for which $T_1(a)= T_3(a) = 0$ and we do not mind
what $T_2(a)$ is. Using the count for $Z(q_2)$ we have:
\begin{eqnarray}
\nonumber F_2(n,0,*,0) &=& \frac{1}{2} \, \#\{ x \in \F_{2^n} \mid T_1(x^5 + x) = 0\}\\
\nonumber  &=& \frac{1}{2} \, Z(q_2)\\
\nonumber  &=& \frac{1}{4} \, (2^n - (\sqrt{2}i)^n - (-\sqrt{2}i)^n - (\sqrt{2}\omega_{8}^3)^n - (\sqrt{2}\omega_{8}^5)^n ).
\end{eqnarray}
Since  $F_2(n,0,*,0) =  F_2(n,0,1,0) + F_2(n,0,0,0)$, we have 
\begin{eqnarray}
\nonumber & & 2^{n-2} - \frac{1}{4}((\sqrt{2}i)^n + (-\sqrt{2}i)^n + (\sqrt{2}\omega_{8}^3)^n + (\sqrt{2}\omega_{8}^5)^n)\\ 
\nonumber &=& F_2(n,0,1,0) + 2^{n-3} -\frac{1}{4}((\sqrt{2}\omega_{8}^3)^n + (\sqrt{2}\omega_{8}^5)^n)
-\frac{1}{8}((\sqrt{2}i)^n + (-\sqrt{2}i)^n)\\
\nonumber &-& \frac{1}{8}((\sqrt{2}\omega_{24}^{5})^n + (\sqrt{2}\omega_{24}^{11})^n + (\sqrt{2}\omega_{24}^{13})^n + (\sqrt{2}\omega_{24}^{19})^n).
\end{eqnarray}
Rearranging gives the formula for $f_2(n,0,1,0)$. 
Finally, by $F_2(n,0,1,0) + F_2(n,0,1,1) = F_2(n,0,1)$, the formula for $f_2(n,0,1,1)$ follows.\qed
\end{proof} 

Our proof for all $n$ is more direct than the amalgamation of the proofs given in~\cite{fitzyucas} ($n$ odd) and~\cite{yucasmullen} ($n$ even), and 
also explains in a simple way why the periodicity arises. Furthermore, for $n$ odd one can derive the trace one formulae from the trace zero formulae 
using~\cite[Lemma 2.6]{fitzyucas}, and thus all of the results of that paper. 
One can also derive the trace one formulae for $n$ odd directly, using the same curve based approach, noting that $T_1(1) = 1$ and
expanding $T_2(x^2 + x + 1)$ and $T_3(x^2 + x +1)$ and adjusting the definitions of $q_1$ and $q_2$ accordingly.

For completeness, we also include the Fourier-analysed formulae for the trace one cases, derived from~\cite{fitzyucas} and~\cite{yucasmullen}.
\begin{proposition}
For $n \ge 3$ we have
\begin{eqnarray}
\nonumber f_2(n,1,0,0) &=& \frac{1}{8}((\sqrt{2}i)^n + (-\sqrt{2}i)^n)
-\frac{1+i}{8}(\sqrt{2}\omega_{8}^3)^n - \frac{1-i}{8}(\sqrt{2}\omega_{8}^5)^n\\
\nonumber & & -\frac{i}{8}((\sqrt{2}\omega_{24}^{5})^n - (\sqrt{2}\omega_{24}^{11})^n
+ (\sqrt{2}\omega_{24}^{13})^n - (\sqrt{2}\omega_{24}^{19})^n)\\
\nonumber f_2(n,1,0,1) &=& -\frac{1}{8}((\sqrt{2}i)^n + (-\sqrt{2}i)^n)
+\frac{1-i}{8}(\sqrt{2}\omega_{8}^3)^n + \frac{1+i}{8}(\sqrt{2}\omega_{8}^5)^n\\
\nonumber & & +\frac{i}{8}((\sqrt{2}\omega_{24}^{5})^n - (\sqrt{2}\omega_{24}^{11})^n
+ (\sqrt{2}\omega_{24}^{13})^n - (\sqrt{2}\omega_{24}^{19})^n)\\
\nonumber f_2(n,1,1,0) &=& -\frac{1}{8}((\sqrt{2}i)^n + (-\sqrt{2}i)^n)
+\frac{1+i}{8}(\sqrt{2}\omega_{8}^3)^n + \frac{1-i}{8}(\sqrt{2}\omega_{8}^5)^n\\
\nonumber & & -\frac{i}{8}((\sqrt{2}\omega_{24}^{5})^n - (\sqrt{2}\omega_{24}^{11})^n
+ (\sqrt{2}\omega_{24}^{13})^n - (\sqrt{2}\omega_{24}^{19})^n)\\
\nonumber f_2(n,1,1,1) &=& \frac{1}{8}((\sqrt{2}i)^n + (-\sqrt{2}i)^n)
-\frac{1-i}{8}(\sqrt{2}\omega_{8}^3)^n - \frac{1+i}{8}(\sqrt{2}\omega_{8}^5)^n\\
\nonumber & & +\frac{i}{8}((\sqrt{2}\omega_{24}^{5})^n - (\sqrt{2}\omega_{24}^{11})^n
+ (\sqrt{2}\omega_{24}^{13})^n - (\sqrt{2}\omega_{24}^{19})^n).
\end{eqnarray}
\end{proposition}

\section{Computing $F_{q}(n,0,0,0)$ using supersingular curves}\label{sec:2r}

In this and the following two sections we shall compute $F_{q}(n,0,0,0)$ and Fourier-analyse the formulae.
Recall that $q = 2^r$ with $r \ge 1$, which we use interchangeably as appropriate for the remainder of the paper. 
As in Sections~\ref{sec:motivation} and~\ref{sec:3coeffsF2}, for $a \in \F_{q^n}$, if $T_1(a) = 0$ then
$a = x^q + x$ for $q$ $x \in \F_{q^n}$. By Lemma~\ref{lem:T2T3}(i) and (ii), we have $T_2(x^q + x) = T_1(x^{q+1} + x^2)$ and
$T_3(x^q + x) = T_1(x^{2q+1} + x^{q+2})$ respectively. Therefore, let 
\begin{eqnarray}
\nonumber q_1 &:& \F_{q^n} \rightarrow \F_q: x \mapsto T_1(x^{q+1} + x^2),\\
\nonumber q_2 &:& \F_{q^n} \rightarrow \F_q: x \mapsto T_1(x^{2q+1} + x^{q+2}).
\end{eqnarray}
We have the following generalisation of Lemma~\ref{lem:fitz}.
As before, for a function $q_i: \F_{q^n} \rightarrow \F_{q}$ let $Z(q_i)$ denote the number of zeros of $q_i$, and
for $w_i \in \F_q$ let $N(w_1,w_2)$ be the number of $x \in \F_{q^n}$ such that $q_i(x) = w_i$, for $i=1,2$.

\begin{lemma}\label{lem:fitzgeneral}
Let $q_1,q_2: \F_{q^n} \rightarrow \F_q$ be any functions. Then 
\[
N(0,0) = \frac{1}{q} \big( Z(q_1) + \sum_{\alpha \in \F_q} Z(\alpha q_1 + q_2) - q^n \big)
\]
\end{lemma}

\begin{proof} 
We have the following identities:
\[
q^n = \sum_{\alpha,\beta \in \F_q} N(\alpha,\beta), \ Z(q_1) =
\sum_{\beta \in \F_q} N(0,\beta), \ \text{and} \
\sum_{\alpha \in  \F_q} Z(\alpha q_1 + q_2) = \sum_{\alpha,\beta \in \F_q} N(\beta, \alpha\beta).
\]
We have
\begin{eqnarray}
\label{gen1_1} \sum_{\alpha,\beta \in \F_q} N(\alpha,\beta) &=& \sum_{\beta \in
  \F_q} N(0,\beta) + \sum_{\beta \in \F_q} \sum_{\alpha \in
  \F_{q}^{\times}} N(\alpha, \beta)\\
 &=& 
 \label{gen1_2} \sum_{\beta \in \F_{q}} N(0,\beta) + \sum_{\beta \in \F_{q}} \sum_{\alpha \in
  \F_{q}^{\times}} N(\alpha, \alpha \beta)\\
\nonumber &=& \sum_{\beta \in \F_{q}} N(0,\beta) + \sum_{\alpha,\beta \in \F_{q}}
N(\beta, \alpha\beta) - q N(0,0),
\end{eqnarray}
where the equality between the second terms on the r.h.s. of lines~(\ref{gen1_1}) and~(\ref{gen1_2}) follow from the
fact that for each $\alpha \in \F_{q}^{\times}$ we have
$\{N(\alpha,\beta)\}_{\beta \in \F_q} = \{N(\alpha,\alpha
  \beta)\}_{\beta \in \F_q}$. Rearranging for $N(0,0)$ gives the stated result.\qed
\end{proof}

Informed by Lemma~\ref{lem:fitzgeneral} we employ the following curves, for $\alpha \in \F_q$:
\begin{eqnarray}
\nonumber C &:& y^q + y = x^{q+1} + x^2,\\
\nonumber C_{\alpha} &:& y^q + y = x^{2q+1} + x^{q+2} + \alpha (x^{q+1} + x^2).
\end{eqnarray}
Observe that for $\alpha \in \F_{q}^{\times}$ we have the $\F_q$-isomorphism $\sigma: C_{\alpha} \longrightarrow C_1	:
(x,y) \mapsto (\alpha x, \alpha^3 y)$ and thus their number of rational points over any extension of $\F_q$ are equal. 
Hence the formula in Lemma~\ref{lem:fitzgeneral} simplifies to:
\begin{equation}\label{simple}
N(0,0) = \frac{1}{q} \big( Z(q_1) + Z(q_2) + (q-1)Z(q_1 + q_2) - q^n \big).
\end{equation}
In the following we therefore focus on these three curves:
\begin{eqnarray}
\nonumber C_1/\F_q &:& y^q + y = x^{q+1} + x^2,\\
\nonumber C_2/\F_q &:& y^q + y = x^{2q+1} + x^{q+2},\\
\nonumber C_3/\F_q &:& y^q + y = x^{2q+1} + x^{q+2} + x^{q+1} + x^2.
\end{eqnarray}
By~(\ref{simple}) one can express $F_{q}(n,0,0,0)$ in terms of the number of points of $C_1,C_2$ and $C_3$ over $\F_{q^n}$.
We shall compute these numbers, which incidentally proves the supersingularity of the three curves (there are many other ways to do so too).
In the rest of this section we detail some preliminary notions and results.

\subsection{Simplifying the point counting on $C_1$, $C_2$ and $C_3$}

In this subsection we show that counting the number of points on the curves $C_1$, $C_2$ and $C_3$ over $\F_{q^n}$ can be simplified by obtaining a partial decomposition of the Jacobian of Artin-Schreier curves.
Let $\overline{\F_2}(x)$ denote the rational function field over $\overline{\F_2}$ which is the algebraic closure of the binary field $\F_2$. An Artin-Schreier curve is 
\[C:y^q+y=f(x)\]
where $q=2^r$ and $f(x)\in\overline{\F_2}(x)$ is not a constant and has no pole of an even order. When $f(x)$ is a polynomial this is equivalent to assuming that $f(x)$ is of an odd degree. If $P_1,P_2,\ldots,P_m$ are the poles of $f(x)$ of order $e_1,e_2,\ldots,e_m$ respectively, then the genus of $C$ is
\begin{equation}\label{genus}
g_C=\frac{q-1}{2}(\sum_{i=1}^m(e_i+1)-2) .
\end{equation}
If $f(x)$ is a polynomial of an odd degree $d$, then the genus of $C$ simply is
\begin{equation}
g_C=\frac{(q-1)(d-1)}{2}.
\end{equation}

We obtain a partial decomposition of the Jacobian of the Artin-Schreier curve $C$ by studying the set of automorphisms of the l.h.s. $y^q + y$ only. Computing the relevant quotient curves and applying a theorem due to Kani and Rosen~\cite{kanirosen} allows one to then infer a relation between the number of points on the original curve and on the quotient curves.

For any positive integer $m$, let $\Tr_{m}: \F_{2^m} \rightarrow \F_{2}$ be the absolute trace function $a \mapsto a + a^{2} + a^{2^2} + \cdots + a^{2^{m-1}}$. We need the following lemma to establish the main result of this subsection. 

\begin{lemma}\label{subgroup}
Let $\alpha$ be a nonzero element of $\F_{2^r}$, and let $H_\alpha=\{x\in \F_{2^r}|\Tr_{r}(\alpha x)=0\}$. Then

\begin{itemize}
\item[(i)] $H_\alpha$ is an additive subgroup of $\F_{2^r}$ when $\F_{2^r}$ is viewed as an additive group,
\item[(ii)] half the elements of $\F_{2^r}$ are in $H_\alpha$, \ie $\#H_\alpha=2^{r-1}$, and
\item[(iii)]$\prod_{\beta\in \F_{2^r}\setminus H_\alpha}\beta=\alpha^{2^{r-1}-1}$.	
\end{itemize}
\end{lemma}	
\begin{proof}
It is easy to check the first part of the claim. To show (ii), notice that the traces of half the elements of $\F_{2^r}$ are zero, and since $\alpha$ is an element of $\F_{2^r}$, the map $\psi:H_\alpha\longrightarrow \F_{2^r}$ taking $x$ to $\alpha x$ induces a bijection between elements of $H_\alpha$ and elements $z$ of $\F_{2^r}$ for which $\Tr_{r}(z)=0$. To prove (iii), we first notice that
\[\Tr_{r}(\alpha x)=\alpha x+(\alpha x)^2+\cdots+(\alpha x)^{2^{r-1}},\]
and hence $\Tr_{r}(\alpha x)$ is a polynomial of degree $2^{r-1}$ which has $2^{r-1}$ roots in $H_\alpha$. Thus we have
\[\Tr_{r}(\alpha x)=\alpha^{2^{r-1}}\prod_{\gamma\in H_\alpha}(x+\gamma).\]
This implies that 
\[\prod_{\stackrel{\gamma\in H_\alpha}{\gamma\neq 0}}\gamma=\alpha^{-2^{r-1}+1}.\]
Now the claim follows from the fact that the product of all non-zero elements of any binary finite field is $1$. \qed
\end{proof}	

Let $\alpha\in\F_{2^r}^{\times}$. It is easy to check that $\alpha$ induces the involution $\phi_\alpha$ of $C$ where $\phi_\alpha$ is given by
\[\phi_\alpha:(x,y)\longrightarrow(x,y+\alpha),\]
and hence the group of automorphisms of $C$, $\mbox{Aut}(C)$, has a subgroup isomorphic to the additive group of $\F_{2^r}$. Thus it follows from Lemma~\ref{subgroup} that $H_\alpha$ can be viewed as a subgroup of $\mbox{Aut}(C)$.
In the following for every $\alpha\in\F_{2^r}$, we compute a plane model for $C/H_\alpha$.

\begin{lemma}
With the notations as above we have
\[C/H_\alpha: y_\alpha^2+\alpha^{2^{r-1}-1}y_\alpha=f(x).
\]
\end{lemma}
\begin{proof}
In order to compute a plane model for $C/H_\alpha$, we need to find two elements in the function field $\overline{\F_2}(x,y)$ which are invariant under the action of $H_\alpha$, and then find an algebraic relation between the two elements. Since $H_\alpha$ is an additive subgroup of $\F_{2^r}$, it follows that $$y_\alpha=\prod_{\gamma\in H_\alpha} (y+\gamma)$$
is invariant under the action of $H_\alpha$. So we choose $y_\alpha$ to be one of the elements that we are looking for and we take the other element to be simply $x$.

Now we claim that
\[y^{q}
+y=y_\alpha^2+\alpha^{2^{r-1}-1}y_\alpha.\]
In order to prove this claim, we need to prove that any element $\delta\in\F_{2^r}$ is a root of the right hand side of the above equation. 
If $\delta\in H_\alpha$, then the claim is trivial. Now suppose that $\delta\in \F_{2^r}\setminus H_\alpha$. First notice that
\[y_\alpha(\delta)+\alpha^{2^{r-1}-1}=\prod_{\gamma\in H_\alpha} (\delta+\gamma) + \alpha^{2^{r-1}-1}.\]
Furthermore since $H_\alpha$ is an additive subgroup of $\F_{2^r}$ of index two, we have
\[\prod_{\gamma\in H_\alpha} (\delta+\gamma)+\alpha^{2^{r-1}-1}=\Big(\prod_{\beta\in \F_{2^r}\setminus H_\alpha}\beta \Big) + \alpha^{2^{r-1}-1},\]
and hence the claim follows from Lemma~\ref{subgroup} part (iii).
\end{proof}

We shall use the following theorem due to Kani and Rosen, which is from Theorem C in~\cite{kanirosen}.

\begin{theorem}\label{Decompose:kanirosen}
Let $C$ be a smooth, projective, absolutely irreducible curve defined over an arbitrary field $K$, and let $H_1,\ldots, H_t < Aut(C)$ be (finite) subgroups with $H_i \cdot H_j = H_j \cdot H_i$ for all $i, j$. Furthermore  let $g_{ij}$ denote the genus of the 
quotient curve $C/(H_i \cdot H_j)$. Then if $g_{ij} = 0$ for $2 \le i<j \le t$ and if
\[
g_C = g_{C/H_2} + \cdots + g_{C/H_t},
\]
then we have (by taking $H_1 = \{1\}$ above) an isogeny of Jacobians:
\[
J_C \sim J_{C/H_2} \times \cdots \times J_{C/H_t}.
\]
\end{theorem}

Applying the above theorem we get the following result.

\begin{theorem}\label{JacobianDecomposition}
Let $C$ denote the Artin-Schrier curve $y^{2^r}+y=f(x)$, and for every $\alpha\in\F_{2^r}^{\times}$ let $C_\alpha$ denote the curve $y_\alpha^2+y_\alpha=\alpha f(x)$. Furthermore let $H_\alpha$ be as in Lemma~\ref{subgroup}. Then we have
\[
J_C \sim \prod_{\alpha\in\F_{2^r}^{\times}}J_{C/H_\alpha}\sim \prod_{\alpha\in\F_{2^r}^{\times}}J_{C_\alpha}.  
\]
\end{theorem}
\begin{proof}
For every pair of distinct elements $\alpha$ and $\beta$ in $\F_{2^r}^{\times}$, we trivially have $H_\alpha\cdot H_\beta=H_\beta\cdot H_\alpha$ and furthermore $(H_\alpha\cdot H_\beta)=\F_{2^r}$ where we are considering $\F_{2^r}$ as an additive group. Notice that for every nonzero $\alpha$, $H_\alpha$ is a subgroup of index two of $\F_{2^r}$. It follows that the genus $g_{\alpha\beta}$ of the curve $C/(H_\alpha\cdot H_\beta)$ is zero. The fact that 
\[g_C=\sum_{\alpha\in\F_{2^r}^{\times}}g_{C/H_\alpha}\] 
follows from~\eqref{genus}. Thus we have
\[
J_C \sim \prod_{\alpha\in\F_{2^r}^{\times}}J_{C/H_\alpha}.
\] 
The rest of the claim follows from the fact that by the change of variable $y_{\alpha} \mapsto \alpha^{2^{r-1} - 1}y_\alpha$, the curve equations for $C/H_{\alpha}$ becomes $C_\alpha:y_\alpha^2 + y_\alpha = \alpha f(x)$.
\end{proof}	

The following is an immediate corollary of the above theorem as the above theorem implies that the L-polynomial of the curve $C$ is equal to product of the L-polynomials of the curves $C_\alpha$.
\begin{corollary}\label{pointcounting-thm}
Let $C$ and $C_\alpha$ be as defined above. Then
\[
\#C(\F_{q^{n}})-\sum_{\alpha\in\F_{q}^{\times}}\#C_\alpha(\F_{q^{n}}) = (q^{n} + 1)(q-1) - 1.
\]
\end{corollary}

\section{Computing the number of rational points on $C_1$, $C_2$ and $C_3$}\label{sec:rationalpoints}
	
With a view to applying Lemma~\ref{lem:fitzgeneral} to compute $F_q(n,0,0,0)$, define 
		\begin{align*}
		q_1(x)&:=x^{2^r+1}+x^2, \\q_2(x)&:=x^{2^{r+1}+1}+x^{2^r+2}, \\q_3(x) &:= q_1(x)+q_2(x),
		\end{align*}
		where $r \geq 1$ is an integer. Let $n \geq 1$ be an integer. In this section, our aim is to find the number of rational points on the curves $$C_{i} : y^{2^r}+y=q_i(x)$$ on $\mathbb F_{2^{rn}}$ for all $i=1,2,3$.
In order to find this we will first try to  find the number of rational points on the curves $$C_{i,\alpha} \: : \: y^2+y=\alpha q_i(x)$$ on $\mathbb F_{2^{rn}}$ for all $\alpha \in \mathbb F_{2^r}^\times$ and for all $i=1,2,3$. We will then apply Corollary~\ref{pointcounting-thm} to find the number of points on $C_i$.

	\begin{proposition} \label{prop-mod-p}
		Let $p$ be an odd prime divisor of $n$. Then $\#C_{i,\alpha}(\mathbb F_{2^{rn}})-\#C_{i,\alpha}(\mathbb F_{2^{r(n/p)}}) \equiv 0 \mod p$ for all $\alpha \in \mathbb F_{2^r}^\times$ and for all $i=1,2,3$.
	\end{proposition}
	\begin{proof}
		We have $\Tr_{rn}(\alpha q_i(x))=p\cdot \Tr_{r(n/p)}(\alpha q_i(x))=\Tr_{r(n/p)}(\alpha q_i(x))$ for all $x \in \mathbb F_{2^{r(n/p)}}$,  for all $\alpha \in \mathbb F_{2^r}^\times$ and for all $i=1,2,3$. Let $\beta \in \mathbb F_{2^{rn}}-\mathbb F_{2^{r(n/p)}}$. Then $\Tr_{rn}(\alpha q_i(x^{2^r}))=\Tr_{rn}(\alpha q_i(x))^{2^r}=\Tr_{rn}(\alpha q_i(x))$ for all $\alpha \in \mathbb F_{2^r}^\times$ and for all $i=1,2,3$. Since order of  $\text{Gal}(\mathbb F_{2^{rn}}/\mathbb F_{2^{r{n/p}}})$ conjugates  of $\beta$, which is $\{\beta,\beta^{2^r},\cdots, \beta^{2^{r(p-1)}}\}$, is $p$, we have  $\#C_{i,\alpha}(\mathbb F_{2^{rn}})-\#C_{i,\alpha}(\mathbb F_{2^{r(n/p)}}) \equiv 0 \mod p$. \qed
	\end{proof}

	The number $\#C_{i,\alpha}(\mathbb F_{2^{rn}})$ can be written as 
	\begin{equation}\label{equation-Lambda}
	\#C_{i,\alpha}(\mathbb F_{2^{rn}})=(2^{rn}+1)+\Lambda_{i,\alpha}(n)2^{(rn+w_{i,\alpha}(n))/2}
	\end{equation} where $w_{i,\alpha}(n)$ is the dimension of  of radical of $\Tr_{rn}(\alpha q_i(x))$ by Proposition \ref{counts} . We call $\Lambda_{i,\alpha}(n) \in \{-1,0,1\}$ as sign of $\#C_{i,\alpha}(\mathbb F_{2^{rn}})$. 
	\begin{corollary} \label{cor-mod-p}
		Let $n \geq 1$ be an integer such that $n=2^um$ where $u \geq 1$ be an integer and $m$ is odd. Assume  $\#C_{i,\alpha}(\mathbb F_{2^{rn}})$ and  dimension of radical of $\Tr_{rn'}(\alpha q_i(x))$ are known where $n'$ equals to $2^um'$ and $m'$ is odd. Then  $\#C_{i,\alpha}(\F_{2^{rn'}})$ is known. 
	\end{corollary}
	\begin{proof}
		We can assume $m=1$ and $m'=p>2$. Since  dimension of radical of $\Tr_{r2^up}(\alpha q_i(x))$ is known, It is enough to find sign of $C_{i,\alpha}(\mathbb F_{r2^up})$ by Equation \ref{equation-Lambda}.  Since $\#C_{i,\alpha}(\mathbb F_{r2^u})$ is known and  $\#C_{i,\alpha}(\mathbb F_{r2^up}) \equiv \#C_{i,\alpha}(\mathbb F_{r2^u})  \mod p$, the sign of $C_{i,\alpha}(\mathbb F_{r2^up})$ is also known. \qed
	\end{proof}


Note that Proposition \ref{minimal-prop} and Corollary \ref{cor-mod-p} imply that in order to find the number of rational points of $C_{i,\alpha}$ on 
$\mathbb F_{2^{rn}}$ for all $n \geq 1$, it is enough to find the number of rational points of $C_{i,\alpha}$ on $\mathbb F_{2^{r2^u}}$ where 
$0 \leq u \leq v$ and $v$ is the minimal integer for which $C(\mathbb F_{2^{r2^v}})$ is minimal, and also the dimension of the radical of 
$\Tr_{rn}(\alpha q_i(x))$ for all $n \geq 1$.

	\subsection{Number of $\mathbb F_{2^{rn}}$-rational points on the curve $C_1 : y^{2^r}+y= q_1(x)$}
In this section we will find the number of rational points on the curves $$C_1 : y^{2^r}+y= x^{2^r+1}+x^2$$ on $\mathbb F_{2^{rn}}$ for all $r,n \in \mathbb Z^+$. In order to find this we will first try to  find the number of rational points on the curves $$C_{1,\alpha} \: : \: y^2+y=\alpha (x^{2^r+1}+x^2)$$ on $\mathbb F_{2^{rn}}$ for all $\alpha \in \mathbb F_{2^r}^\times$, and we will then apply Corollary~\ref{pointcounting-thm}. We will see that the number of points is the same for each $\alpha$.

		\begin{lemma} \label{C1-dimension}
					 	Let $B_1(x,y)=\Tr_{rn}(q_1(x+y)+q_1(x)+q_1(y))$ and let the radical of $B_1(x,y)$ be $W_1(n)=\{x \in \mathbb F_{2^{rn}} \: | \: B_1(x,y)=0 \: \text{ for all } \: y \in \mathbb F_{2^{rn}} \}$. Denote the dimension of radical of $B_1(x,y)$ over $\mathbb F_{2^{rn}}$ as $w_r(n)$, then we have 
					 $$w_r(n)=
					 \begin{cases}
					 r & \text{if $n$ is odd}\\
					 2r & \text{if $n$ is even.}\\
					 \end{cases}
					 $$
		\end{lemma}
		\begin{proof}
				Let $B_1(x,y)=\Tr_{rn}(q_1(x+y)+q_1(x)+q_1(y))$. Then we have $$B_1(x,y)=\Tr_{rn}(y^{2^r}(x^{2^{2r}}+x))$$ and the radical of $B_1(x,y)$ is $$W_1(n)=\{x \in \mathbb F_{2^{rn}} \: | \: B_1(x,y)=0 \: \text{ for all } \: y \in \mathbb F_{2^{rn}} \}= \mathbb F_{2^{rn}} \cap \mathbb F_{2^{2r}}=\mathbb F_{2^{r(2,n)}}.$$ \qed
\end{proof}
		\begin{lemma}\label{lemma-W2n-C1}
			Let the notation be above. $\Tr_{2r}(q_1(W_1(2)))=\mathbb F_{2}$ and $\Tr_{2r}(q_1(W_1(n)))=0$ for $4 \mid n$.
		\end{lemma}
		\begin{proof}
		Suppose $n=2$. Since $x^{2^r+1} \in \mathbb F_{2^r}$ for all $x \in \mathbb F_{2^{2r}}$ we have $\Tr_{2r}(x^{2^r+1})=0$ for all $x \in \mathbb F_{2^{2r}}$. Therefore $$\Tr_{2r}(x^{2^r+1}+x^2)=\Tr_{2r}(x).$$
		Now suppose $4|n$. Since $W_1(n)=\mathbb F_{2^{2r}}$, for all $x \in W_1(n)$ we have $$\Tr_{rn}(x^{2^r+1}+x^2)=0$$ for all $x \in W_1(n)$.\qed	
		\end{proof}
		\begin{theorem}\label{thm-points on curve 1}
			 If $r$ is odd, then we have $$ \#C_{1,\alpha}(\mathbb F_{2^{rn}})=
			 \begin{cases}
			 (2^{rn}+1)+2^{r(n+1)/2} &\text{if $n \equiv 1,7 \mod 8$,}\\
			 (2^{rn}+1)-2^{r(n+1)/2} &\text{if $n \equiv 3,5 \mod 8$,}\\
			 (2^{rn}+1) &\text{if $n \equiv 2,6 \mod 8$,}\\
			 (2^{rn}+1)+2^{r(n+2)/2} &\text{if $n \equiv 4 \mod 8$,}\\
			 (2^{rn}+1)-2^{r(n+2)/2} &\text{if $n \equiv 0 \mod 8$}
			 \end{cases}
			 $$
			 if $r$ is even, then we have $$ \#C_{1,\alpha}(\mathbb F_{2^{rn}})=
			 \begin{cases}
			 (2^{rn}+1)+2^{r(n+1)/2} &\text{if $n \equiv 1,3 \mod 4$,}\\
			 (2^{rn}+1) &\text{if $n \equiv 2 \mod 4$,}\\
			 (2^{rn}+1)-2^{r(n+2)/2} &\text{if $n \equiv 0\mod 4$.}\\
			 \end{cases}
			 $$
	\end{theorem}
		
	\begin{proof}
		In order to prove this theorem, we will find the number of rational points of $C_{1,1}$ on $\mathbb F_{2^{r}}$, $\mathbb F_{2^{2r}}$, $\mathbb F_{2^{4r}}$ and $\mathbb F_{2^{8r}}$. Then we will apply Proposition \ref{minimal-prop} and Proposition \ref{prop-mod-p} for finding the number of rational points of $C_{1,1}$ on $\mathbb F_{2^{rn}}$ for any $n \geq 1$. \\
		
		Since $x^{2^r+1}+x^2=0$ for all $x \in \mathbb F_{2^r}$, the number of rational points of $C_{1,1}$ on $\mathbb F_{2^{r}}$ is $(2^r+1)+2^r$. Therefore, the number of rational points of $C_{1,1}$ on $\mathbb F_{2^{rn}}$ is $(2^{rn}+1)+2^{r(n+1)/2}$ if $n \equiv 1,7 \mod 8$ and the number of rational points of $C_{1,1}$ on $\mathbb F_{2^{r}}$ is $(2^{rn}+1)+2^{r(n+1)/2}$ if $n \equiv 3,5 \mod 8$ by Proposition \ref{prop-mod-p} and Lemma \ref{C1-dimension}. \\
		
		Since $\Tr_{2r}(q_1(W_1(2)))=\mathbb F_2$ by Lemma \ref{lemma-W2n-C1}, we have  the number of rational points of $C_{1,1}$ on $\mathbb F_{2^{2r}}$ is $2^{2r}+1$.  Therefore, the number of rational points of $C_{1,1}$ on $\mathbb F_{2^{rn}}$ is $2^{rn}+1$ if $n \equiv 2,6 \mod 8$ by Proposition \ref{prop-mod-p} and Lemma \ref{C1-dimension}. \\
		
		Since $\Tr_{4r}(q_1(W_1(4)))= 0$ by Lemma \ref{lemma-W2n-C1}, we have  the number of rational points of $C_{1,1}$ on $\mathbb F_{2^{2r}}$ is $(2^{2r}+1) \pm 2^{r(n+2)/2}$. Since genus of $C_{1,1}$ is $2^{r-1}$ we have $C_1(\mathbb F_{2^{4r}})$ is maximal or minimal by Proposition \ref{minimal-prop}.  Since we cannot write $1$ as a $\mathbb Q$ linear combination of $\{w \in \mathbb C \: | \: w^4=-1\}$ by Lemma \ref{lemma-unity}, and since we cannot write $\sqrt 2$ as a $\mathbb Q$ linear combination of $\{w \in \mathbb C \: | \: w^4=1\}$, we have $C_{1,1}(\mathbb F_{2^{4r}})$ is maximal and $C_{1,1}(\mathbb F_{2^{8r}})$ is minimal by Proposition \ref{minimal-prop} if $r$ is odd, and $C_{1,1}(\mathbb F_{2^{4r}})$ is minimal if $r$ is even. Therefore if $r$ is odd, the number of rational points of $C_{1,1}$ on $\mathbb F_{2^{rn}}$ is $(2^{rn}+1)+2^{r(n+2)/2}$ if $n \equiv 4 \mod 8$ and $(2^{rn}+1)-2^{r(n+2)/2}$ if $n \equiv 0 \mod 8$; if $r$ is even, the number of rational points of $C_{1,1}$ on $\mathbb F_{2^{rn}}$ is $(2^{rn}+1)-2^{r(n+2)/2}$ if $n \equiv 0,4 \mod 8$  by Proposition \ref{prop-mod-p} and Lemma \ref{C1-dimension}.\\

			Note that since the map $(x,y) \to (\alpha^{-1/2}x,y)$ gives an isomorphism between $C_{1,\alpha} \to C_{1,1}$ where $\alpha \in \mathbb F_{2^r}^\times$, their number of rational points equal in any extension of $\mathbb F_{2^r}$. \qed
	\end{proof}

Using Corollary~\ref{pointcounting-thm} and Theorem~\ref{thm-points on curve 1}, Table~\ref{table:C1points} in the appendix presents the number of 
$\F_{2^{rn}}$-rational points on $C_1:y^{2^r}+y=q_1(x)$.

\subsection{Number of $\mathbb F_{2^{rn}}$-rational points on the curve  $C_2 : y^{2^r}+y= q_2(x)$}
In this section we will find the number of $\mathbb F_{2^{rn}}$-rational points on the curve $$C_2:y^{2^r}+y= x^{2^{r+1}+1}+x^{2^r+2}$$ for all $r,n \in \mathbb Z^+$. In order to find this we will first try to  find the number of rational points on the curves $$C_{2,\alpha} \: : \: y^2+y=\alpha (x^{2^{r+1}+1}+x^{2^r+2})$$ on $\mathbb F_{2^{rn}}$ for all $\alpha \in \mathbb F_{2^r}^\times$,
	and we will then apply Corollary~\ref{pointcounting-thm}.
	We will see that the number of points depends on whether $\alpha$ is a cube, for odd $r$.	
	
	\begin{lemma} \label{lemma-C2-dimensions}
		Let $B_{2,\alpha}(x,y)=\Tr_{rn}(\alpha(q_2(x+y)+q_2(x)+q_2(y)))$ where $\alpha \in \mathbb F_{2^r}$ and let the radical of $B_{2,\alpha}(x,y)$ be $W_1(n)=\{x \in \mathbb F_{2^{rn}} \: | \: B_{2,\alpha}(x,y)=0 \: \text{ for all } \: y \in \mathbb F_{2^{rn}} \}$. Denote the dimension of radical of $B_1(x,y)$ over $\mathbb F_{2^{rn}}$ as $w_r(n)$. If $\alpha \in \{x^3 \: | \: x \in \mathbb F_{2^r} \}$, then we have $$w_{r,\alpha}(n)=
		\begin{cases}
		r & \text{if $n$ is odd},\\
		2r & \text{if $n \equiv 2 \mod 4$,}\\
		2r+2 & \text{if $n \equiv 0 \mod 4$.}\\
		\end{cases}
		$$ If $\alpha \not \in \{x^3 \: | \: x \in \mathbb F_{2^r} \}$, then we have $$w_{r,\alpha}(n)=
		\begin{cases}
		r & \text{if $(n,6)=1$},\\
		2r & \text{if $(n,6)=2$},\\
		r+2 & \text{if $(n,6)=3$},\\
		2r+2 & \text{if $(n,6)=6$}.\\
		\end{cases}
		$$
	\end{lemma}
	\begin{proof}
			Let $B_{2,\alpha}(x,y)=\Tr_{rn}(\alpha(q_2(x+y)+q_2(x)+q_2(y)))$ where $\alpha \in \mathbb F_{2^r}$. Then we have $$B_{2,\alpha}(x,y)=\Tr_{rn}(\alpha y^{2^{r+1}}[(x+\alpha x^4)+(x+\alpha x^4)^{2^{2r}}])$$ and the radical of  $B_{2,\alpha}(x,y)$ is $$W_{2,\alpha}(n)=\{ x \in \mathbb F_{2^{rn}} \: | \: (x+\alpha x^4)+(x+\alpha x^4)^{2^{2r}}=0 \}.$$ This can be rewritten as (multiplying by $\alpha^3$ and transforming $\alpha x \to x$) $$W_{2,\alpha}(n)=\{ x \in \mathbb F_{2^{rn}} \: | \: (x^4+\alpha^2 x)+(x^4+\alpha^2 x)^{2^{2r}}=0 \in \mathbb F_{2^{2r}} \}.$$ 
			Let $x^4+\alpha^2 x=x(x+a)(x+b)(x+c)$ where $a,b,c \in \overline{\mathbb F_{2^r}}$. Then $$(x^4+\alpha^2 x)+(x^4+\alpha^2 x)^{2^{2r}}=(x^{2^{2r}}+x)(x^{2^{2r}}+x+a)(x^{2^{2r}}+x+b)(x^{2^{2r}}+x+c).$$ Therefore $w_{r,\alpha}(n)=\frac12\#\{u \: : \: u \in \{0,a,b,c\} \text{ and } u \in \mathbb F_{2^{rn}} \}+r(n,2).$ Moreover if $\alpha^2$ (or $\alpha$) is not a cube in $\mathbb F_{2^r}$ then $x^3+\alpha^2$ is irreducible  over $\mathbb F_{2^r}$ and hence all roots of $x^3+\alpha^2$ are in $\mathbb F_{2^{3r}}$,  if $\alpha^2$ (or $\alpha$) is a cube in $\mathbb F_{2^r}$ then $x^{4r}+x=0$ for all $x \in W_{2,\alpha}(n)$  and hence all roots of $x^3+\alpha^2$ are in $\mathbb F_{2^{4r}}$. \qed
	\end{proof}
\begin{lemma} \label{C2-W2}
	Let notation be as above. Then we have $\Tr_{2r}(\alpha q_2(W_{2,\alpha}(n)))=0$ and therefore $\#C_{2,\alpha}(\mathbb F_{2^{2r}})-(2^{2r}+1) \ne 0$.
\end{lemma}
\begin{proof}
		Since $\alpha(x^{2^{r+1}+1}+x^{2^{r+2}}) \in \mathbb F_{2^r}$ for all $x \in \mathbb F_{2^{2r}}$ and for all $\alpha \in \mathbb F_{2^r}$, we have $$\Tr_{2r}(\alpha(x^{2^{r+1}+1}+x^{2^{r+2}}))=0$$ for all $x \in \mathbb F_{2^{2r}}$ and for all $\alpha \in \mathbb F_{2^r}$. Therefore $\#C_{2,\alpha}(\mathbb F_{2^{2r}})-(2^{2r}+1) \ne 0$. \qed
\end{proof}

\begin{theorem}\label{thm-points on curve 2}
If $r$ is odd, then we have
$$ \#C_{2,\alpha}(\mathbb F_{2^{rn}})=
\begin{cases}
(2^{rn}+1)+2^{r(n+1)/2} &\text{if $n \equiv 1,7 \mod 8$},\\
(2^{rn}+1)-2^{r(n+1)/2} &\text{if $n \equiv 3,5 \mod 8$},\\
(2^{rn}+1)+2^{r(n+2)/2} &\text{if $n \equiv 2,6 \mod 8$},\\
(2^{rn}+1) &\text{if $n \equiv 4 \mod 8$}\\
(2^{rn}+1)-2^{r(n+2)/2+1} &\text{if $n \equiv 0 \mod 8$}.\\
\end{cases}
$$
 If $r$ is even and $\alpha \in \{x^3 \: | \: x\in \mathbb F_{2^r} \}$, , then we have
$$ \#C_{2,\alpha}(\mathbb F_{2^{rn}})=
\begin{cases}
(2^{rn}+1)+2^{r(n+1)/2} &\text{if $n \equiv 1,3 \mod 4$},\\
(2^{rn}+1)+2^{r(n+2)/2} &\text{if $n \equiv 2 \mod 4$},\\
(2^{rn}+1)-2^{r(n+2)/2+1} &\text{if $n \equiv 0 \mod 4$}.\\
\end{cases}
$$
If $r$ is even and $\alpha \not \in \{x^3 \: | \: x\in \mathbb F_{2^r} \}$, then we have $$ \#C_{2,\alpha}(\mathbb F_{2^{rn}})=
\begin{cases}
(2^{rn}+1)+2^{r(n+1)/2} & \text{if $(n,6)=1$},\\
(2^{rn}+1)+2^{r(n+2)/2} & \text{if $(n,6)=2$},\\
(2^{rn}+1)-2^{r(n+1)/2+1} & \text{if $(n,3)=3$},\\
(2^{rn}+1)-2^{r(n+2)/2+1} & \text{if $(n,6)=6$}.\\
\end{cases}
$$
\end{theorem}
\begin{proof} 
		We will prove this theorem in same manner of Theorem \ref{thm-points on curve 1}. Note that the genus of $C_{2,\alpha}$ is $2^{r}$ for all $\alpha \in \mathbb F_{2^r}^\times$.\\
		
		Assume $r$ is odd. Suppose $C_{2,\alpha}(\mathbb F_{2^{4r}})$ is maximal. Then it requires $\#C_{2,\alpha}(\mathbb F_{2^{2r}})-(2^{2r}+1) = 0$, but it is not possible by Lemma \ref{C2-W2}. Suppose $C_{2,\alpha}(\mathbb F_{2^{4r}})$ is minimal curve. Then we cannot write $\sqrt 2$ as a $\mathbb Q$ linear combination of $\{\eta \in \mathbb C \: | \: \eta^4=1\}$. Therefore $\#C_{2,\alpha}(\mathbb F_{2^{4r}})-(2^{4r}+1) = 0$. Furthermore, since $W_{1,\alpha}(n) \subset \mathbb F_{2^{4r}}$, we have $C_{2,\alpha}(\mathbb F_{2^{8r}})$ is maximal or minimal. Since  $\sqrt 2$ cannot be written as a $\mathbb Q$ linear combination of $\{\eta \in \mathbb C \: | \: \eta^8=-1\}$ by Lemma \ref{lemma-unity}, the curve $C_{2,\alpha}(\mathbb F_{2^{8r}})$ must be minimal.\\

		Assume $r$ is even. Then $\alpha q_2(wx)=\alpha q_2(w^2x)=\alpha q^2(x)$ for all $x \in \mathbb F_{2^{rn}}$ and for all $\alpha \in \mathbb F_{2^r}^\times$ where $w^2+w+1=0$ and therefore $\#C_{2,\alpha}(\mathbb F_{2^{rn}}) \equiv 0 \mod 3$. Since we know the dimensions $w_{r,\alpha}(n)$ for all $\alpha \in \mathbb F_{2^r}^\times$ and $n \geq 1$, one can easily find the signs of $C_{2,\alpha}(\mathbb F_{2^{rn}})$  for all $\alpha \in \mathbb F_{2^r}^\times$ and $n \geq 1$. \qed
\end{proof}
		
Using Corollary~\ref{pointcounting-thm} and Theorem~\ref{thm-points on curve 2}, Table~\ref{table:C2points} in the appendix presents the number of $\F_{2^{rn}}$-rational points on $C_2:y^{2^r}+y=q_2(x)$.

We remark that Moisio and Ranto~\cite{moisioranto} have counted $F_q(n,0,*,t_3)$ for $t_3 \in \F_{q} = \F_{2^r}$.
Their methods are different, and do not yield the same type of general formula as here.
	
\subsection{Number of $\mathbb F_{2^{rn}}$-rational points on the curve $C_3 : y^{2^r}+y= q_3(x)$}
In this section we will find the number of rational points on the curve $$C_3 : y^{2^r}+y= x^{2^{r+1}+1}+x^{2^r+2}+x^{2^r+1}+x^2$$ on $\mathbb F_{2^{rn}}$ for all $r,n \in \mathbb Z^+$. In order to find this we will first try to  find the number of rational points on the curves $$C_{3,\alpha} \: : \: y^2+y=\alpha (x^{2^{r+1}+1}+x^{2^r+2}+x^{2^r+1}+x^2)$$ on $\mathbb F_{2^{rn}}$ for all $\alpha \in \mathbb F_{2^r}^\times$.
	\begin{lemma}
		Let $B_{3,\alpha}(x,y)=\Tr_{rn}(\alpha(q_3(x+y)+q_3(x)+q_3(y)))$ where $\alpha \in \mathbb F_{2^r}$ and let the radical of $B_{3,\alpha}(x,y)$ be $W_{3,\alpha}(n)=\{x \in \mathbb F_{2^{rn}} \: | \: B_{3,\alpha}(x,y)=0 \: \text{ for all } \: y \in \mathbb F_{2^{rn}} \}$. Denote the dimension of radical of $B_{3,\alpha}(x,y)$ over $\mathbb F_{2^{rn}}$ as $w_r(n)$. If  $x^3+x+\alpha^{-1}$ has one root or three roots in $\mathbb F_{2^r}$, then we have $$w_{r,\alpha}(n)=
		\begin{cases}
		r & \text{if $n$ is odd},\\
		2r & \text{if $n \equiv 2 \mod 4$},\\
		2r+2 & \text{if $n \equiv 0 \mod 4$}.\\
		\end{cases}
		$$
		If  $x^3+x+\alpha^{-1}$ has no root in $\mathbb F_{2^r}$ $$w_{r,\alpha}(n)=
		\begin{cases}
		r & \text{if $(n,6)=1$},\\
		2r & \text{if $(n,6)=2$},\\
		r+2 & \text{if $(n,6)=3$},\\
		2r+2 & \text{if $(n,6)=6$}.\\
		\end{cases}
		$$
	\end{lemma}
	\begin{proof}
		Let $B_{3,\alpha}(x,y)=\Tr_{rn}(\alpha(q_3(x+y)+q_3(x)+q_3(y)))$ where $\alpha \in \mathbb F_{2^r}$. Then we have $$B_{3,\alpha}(x,y)=\Tr_{rn}(\alpha^2 y^{2^{r+1}}[(\alpha^{-1}x+ x^2+ x^4)+(\alpha^{-1}x+ x^2+x^4)^{2^{2r}}])$$ and the radical of  $B_{3,\alpha}(x,y)$ is $$W_{3,\alpha}(n)=\{ x \in \mathbb F_{2^{rn}} \: | \: (\alpha^{-1}x+ x^2+ x^4)+(\alpha^{-1}x+ x^2+x^4)^{2^{2r}}=0 \}.$$ Let $x^4+x^2+\alpha^{-1} x=x(x+a)(x+b)(x+c)$ where $a,b,c \in \overline{\mathbb F_{2^r}}$. Then $$(x^4+x^2+\alpha^{-1} x)+(x^4+x^2+\alpha^{-1} x)^{2^{2r}}=(x^{2^{2r}}+x)(x^{2^{2r}}+x+a)(x^{2^{2r}}+x+b)(x^{2^{2r}}+x+c).$$ Therefore $w_{r,\alpha}(n)=\frac12\#\{u \: : \: u \in \{0,a,b,c\} \text{ and } u \in \mathbb F_{2^{rn}} \}+r(n,2).$ Moreover if  $x^3+x+\alpha^{-1}$ is irreducible over $\mathbb F_{2^r}$, then all roots of $x^3+x+\alpha^{-1}$ are in $\mathbb F_{2^{3r}}$,  if $x^3+x+\alpha^{-1}$ has one root in $\mathbb F_{2^r}$ and two roots in $\mathbb F_{2^{2r}}$; or three roots in $\mathbb F_{2^r}$ then $x^{4r}+x=0$ for all $x \in W_{2,\alpha}(n)$  and hence all roots of $x^3+x+\alpha^{-1}$ are in $\mathbb F_{2^{4r}}$. \qed	
	\end{proof}
	
\begin{lemma}\label{onto-radical}
	Let $x^{2^{2r}}+x+b=0$ where $b \in \mathbb F_{2^{2r}}$. Then $$\Tr_{4r}(\alpha(x^{2^{r+1}+1}+x^{2^r+2}+x^{2^r+1}+x^2))=\begin{cases}0 \text { if } b \in \mathbb F_{2^{2r}} \\ 1 \text{ if } b \in \mathbb F_{2^{2r}} \backslash \mathbb F_{2^r} \end{cases}$$  where $\alpha \in \mathbb F_{2^r}^\times$. Therefore, if $x^3+x+\alpha^{-1}$ has zero root or three roots in $\mathbb F_{2^r}$, then $\Tr_{4r}(\alpha q_3(W_{3,\alpha}))=0$ and if $x^3+x+\alpha^{-1}$ has one root in $\mathbb F_{2^r}$, then $\Tr_{4r}(\alpha q_3(W_{3,\alpha}))=\mathbb F_2$.
\end{lemma}
\begin{proof}
	We expand the trace 
	\begin{align*}
	&\Tr_{4r}(\alpha(x^{2^{r+1}+1}+x^{2^r+2}+x^{2^r+1}+x^2))\\
	=&\Tr_{2r}(\alpha(x^{2^{r+1}+1}+x^{2^r+2}+x^{2^r+1}+x^2)+(\alpha(x^{2^{r+1}+1}+x^{2^r+2}+x^{2^r+1}+x^2)^{2^{2r}})\\
	=&\Tr_{2r}(\alpha((x^{2^{r+1}+1}+(x+b)^{2^{r+1}+1})+(x^{2^{r}+2}+(x+b)^{2^{r}+2})+(x^{2^{r}+1}+(x+b)^{2^{r}+1})+b^2))\\
	=&\Tr_{2r}(\alpha b^2)+\Tr_{2r}(\alpha((b^{2^{r+1}+1}+b^{2^r+2}))+\Tr_{2r}(\alpha b^{2^r+1})\\
	+&\Tr_{2r}(bx^{2^{r+1}}+b^{2^r}x^2)+\Tr_{2r}(bx^{2^{r}}+b^{2^r}x)+\Tr_{2r}(b^2x^{2^{r}}+b^{2^{r+1}}x)\\
	=&\Tr_{2r}(\alpha b^2)\\
	=&\begin{cases}0 \text { if } b \in \mathbb F_{2^{2r}} \\ 1 \text{ if } b \in \mathbb F_{2^{2r}} \backslash \mathbb F_{2^r}. \end{cases}
	\end{align*}
		If $f_{\alpha}$ has three roots in $\mathbb F_{2^r}$ (resp. one root in $\mathbb F_{2^r}$ and two roots in $\mathbb F_{2^{2r}}$), then $$(\alpha^{-1}x+ x^2+ x^4)+(\alpha^{-1}x+ x^2+x^4)^{2^{2r}}=(x^{2^{2r}}+x)(x^{2^{2r}}+x+a)(x^{2^{2r}}+x+b)(x^{2^{2r}}+x+c)$$ where $a,b,c \in \F_{2^r}$ (resp. $a \in \F_{2^r}$ and $b,c \in \mathbb F_{2^{2r}}$). 	If $f_{\alpha}$ has no root in $\mathbb F_{2^r}$, then $W_{3,\alpha}(4)=\mathbb F_{2^{2r}}$. \qed
\end{proof}

\begin{theorem}\label{thm:C3 points}
	 If $r$ is odd and $x^3+x+\alpha^{-1}$ has three roots in $\mathbb F_{2^r}$, then we have
	 $$ \#C_{3,\alpha}(\mathbb F_{2^{rn}})=
	 \begin{cases}
	 (2^{rn}+1)+2^{r(n+1)/2} &\text{if $n \equiv 1,7 \mod 8$},\\
	 (2^{rn}+1)-2^{r(n+1)/2} &\text{if $n \equiv 3,5 \mod 8$},\\
	 (2^{rn}+1) &\text{if $n \equiv 2 \mod 8$},\\
	 (2^{rn}+1)+2^{r(n+2)/2+1} &\text{if $n \equiv 4 \mod 8$},\\
	 (2^{rn}+1)-2^{r(n+2)/2+1} &\text{if $n \equiv 0 \mod 8$}.\\
	 \end{cases}
	 $$
 If $r$ is even and $x^3+x+\alpha^{-1}$ has three roots in $\mathbb F_{2^r}$, then we have
	 $$ \#C_{3,\alpha}(\mathbb F_{2^{rn}})=
	 \begin{cases}
	 (2^{rn}+1)+2^{r(n+1)/2} &\text{if $n \equiv 1,3,5,7 \mod 8$},\\
	 (2^{rn}+1) &\text{if $n \equiv 2 \mod 8$},\\
	 (2^{rn}+1)-2^{r(n+2)/2+1} &\text{if $n \equiv 4 \mod 8$},\\
	 (2^{rn}+1)-2^{r(n+2)/2+1} &\text{if $n \equiv 0 \mod 8$}.\\
	 \end{cases}
	 $$
If $r$ is odd and $x^3+x+\alpha^{-1}$ has one root in $\mathbb F_{2^r}$,	 then we have
	 $$ \#C_{3,\alpha}(\mathbb F_{2^{rn}})=
	 \begin{cases}
	 (2^{rn}+1)+2^{r(n+1)/2} &\text{if $n \equiv 1,7 \mod 8$},\\
	 (2^{rn}+1)-2^{r(n+1)/2} &\text{if $n \equiv 3,5 \mod 8$},\\
	 (2^{rn}+1) &\text{if $n \equiv 2,4 \mod 8$},\\
	 (2^{rn}+1)-2^{r(n+2)/2+1} &\text{if $n \equiv 0 \mod 8$}.\\
	 \end{cases}
	 $$
If $r$ is even and $x^3+x+\alpha^{-1}$ has one root in $\mathbb F_{2^r}$, then we have
	 $$ \#C_{3,\alpha}(\mathbb F_{2^{rn}})=
	 \begin{cases}
	 (2^{rn}+1)+2^{r(n+1)/2} &\text{if $n \equiv 1,3,5,7 \mod 8$},\\
	 (2^{rn}+1) &\text{if $n \equiv 2,4 \mod 8$},\\
	 (2^{rn}+1)-2^{r(n+2)/2+1} &\text{if $n \equiv 0 \mod 8$}.\\
	 \end{cases}
	 $$
If  $x^3+x+\alpha^{-1}$ has no root in $\mathbb F_{2^r}$, the possibilities for $\#C_{3,\alpha}(\mathbb F_{2^{rn}})$ are given by 
Table~\ref{table:C3pointsnoroots} in the appendix.
\end{theorem}

\begin{proof}
		We will prove this theorem in same manner of Theorem \ref{thm-points on curve 1} and we will also use Corollary \ref{cor-mod-p} for the case $x^3+x+\alpha^{-1}$ has no root in $\mathbb F_{2^r}$. Note that the genus of $C_{3,\alpha}$ is $2^{r}$ for all $\alpha \in \mathbb F_{2^r}^\times$.\\
		
		If $x^3+x+\alpha^{-1}$ has one root or three roots in $\mathbb F_{2^r}$ then $C_{3,\alpha}(\mathbb F_{2^{8r}})$ is maximal or minimal and 	if $x^3+x+\alpha^{-1}$ has no root in $\mathbb F_{2^r}$ then $C_{3,\alpha}(\mathbb F_{2^{24r}})$ is maximal or minimal. Since $\sqrt {2^r}$ (or $\sqrt {2^{3r}}$) cannot be written as a $\mathbb Z$-linear combination of $\{\eta \: | \: \eta^8=-1\}$ by Lemma \ref{lemma-unity}, if $x^3+x+\alpha^{-1}$ has one root or three roots in $\mathbb F_{2^r}$ then $C_{3,\alpha}(\mathbb F_{2^{8r}})$ is minimal and if $x^3+x+\alpha^{-1}$ has zero root in $\mathbb F_{2^r}$ then $C_{3,\alpha}(\mathbb F_{2^{24r}})$ is minimal.\\
		
		If $x^3+x+\alpha^{-1}$ has one root in $\mathbb F_{2^r}$ then  $\#C_{3,\alpha}(\mathbb F_{2^{4r}})-(2^{4r}+1)=0$, if $x^3+x+\alpha^{-1}$ has  three roots in $\mathbb F_{2^r}$ and $r$ is odd (even) then $C_{3,\alpha}(\mathbb F_{2^{4r}})$ is maximal (minimal), if $x^3+x+\alpha^{-1}$ has zero root in $\mathbb F_{2^r}$ and $r$ is odd (even) then $C_{3,\alpha}(\mathbb F_{2^{12r}})$ is maximal (minimal) by Lemma \ref{onto-radical}. \qed
\end{proof}

\begin{lemma}\label{quadratic-solutions}
	A quadratic equation $x^2+bx+a= 0$, $a \in \mathbb F_{2^r}$ and $b \in \mathbb F_{2^r}^\times$, has two distinct roots in $\mathbb F_{2^r}$
	if and only if $\Tr_{r}(a/b^2)= 0$.
\end{lemma}
\begin{proof}
	Since $b \ne 0$, we can divide the equation by $b^2$. So we get $(x/b)^2+(x/b)=a/b^2$. Hence  $\Tr_{r}(a/b^2)= 0$ if and only if $x/b \in \mathbb F_{2^r}$ if and only if $x \in \mathbb F_{2^r}$. \qed
\end{proof}
\begin{lemma}\label{M013}
	Denote $f_{\beta}(x) = x^3 + x + \beta$, where $\beta \in \mathbb F_{2^r}^\times$. Let 
	$$M_i = \#\{f_{\beta} \: | \: f_{\beta}(x) \text{ has precisely $i$ roots in } \mathbb F_{2^r}\}.$$ Then we have $(M_3,M_1,M_0)=\big(\frac{2^{r-1}-1}{3},2^{r-1}-1,\frac{2^r+1}{3}\big)$ if $r$ is odd, $(M_3,M_1,M_0)=\big(\frac{2^{r-1}-2}{3},2^{r-1},\frac{2^r-1}{3}\big)$ if $r$ is even.	
\end{lemma}
\begin{proof}
	The derivative of $f_{\beta}(x)$ is $x^2+1$ for all $\beta \in \mathbb F_{2^r}^\times$. Thus it has multiple roots if and only if $1$ is a root of $f_\beta$. Moreover $x=1$ is root of $f_\beta$ if $\beta=0$ (which is not the case) and the other root of $f_0$ is $0$.\\
	Let $F^{\times\times} := \mathbb F_{2^r}\backslash\{0,1\}$. Now assume $f_\beta(x)$ has two different roots $a,b \in F^{\times\times}$. This means $f_\beta(a)=f_\beta(b)$. It follows that $(a+b)(b^2+ab+(a^2+1))=0$; that is, $b^2+ab+(a^2+1)=0$. By Lemma \ref{quadratic-solutions}, $\Tr_{r}(\frac{a^2+1}{a^2})=\Tr_{r}(\frac1a)+\Tr_{r}(1)$ must to be $0$.\\
	Note that $1/a$ is a permutation on the set $F^{\times\times}$ and half of $\mathbb F_{2^r}$ have trace $0$ and the other half of $\mathbb F_{2^r}$ have trace $1$. Therefore we have $M_3=[2^{r-1}-\#\{x \in \{0,1\} \: |\Tr_{r}(x) = \Tr_{r}(1)\}]/3$ and $M_1=(2^r-2)-3M_3$ and $M_0=(2^r-1)-M_3-M_1$.
\qed
\end{proof}

Using Corollary~\ref{pointcounting-thm} and Theorem~\ref{thm:C3 points}, and Lemma~\ref{M013}, Table~\ref{table:C3points} in the appendix presents the number of $\F_{2^{rn}}$-rational points on $C_3:y^{2^r}+y=q_3(x)$.

\section{Explicit formulae for $F_{q}(n,0,0,0)$}\label{sec:explicit}

In this section we give an explicit formula for $F_{q}(n,0,0,0)$. In addition we present
Fourier-analysed formula for the number of points on the curves $C_1$, $C_2$ and $C_3$, and thus for $F_{q}(n,0,0,0)$. 
We use the former to infer the characteristic polynomial of Frobenius of these curves and hence the decomposition of their Jacobians.

Using Lemma~\ref{lem:fitzgeneral} and Eq.~(\ref{simple}), we now combine the results contained in Tables~\ref{table:C1points},~\ref{table:C2points} 
and~\ref{table:C3points} to give the formula for $F_{q}(n,0,0,0)$, presented in Table~\ref{table:Fqn000} in the appendix.



\subsection{$C_{1}/\F_q: y^{q}+y = x^{q+1} + x^2$}
By Fourier-analysing the formulae in Table~\ref{table:C1points}, for $r$ odd we have:
\[
\#C_1(\F_{q^n}) = q^n + 1 -\frac{(q-1)(q-\sqrt{2q})}{4}((\sqrt{q}\omega_8)^n + (\sqrt{q}\omega_{8}^7)^n) 
-\frac{(q-1)(q + \sqrt{2q})}{4}((\sqrt{q} \omega_{8}^3)^n + (\sqrt{q}\omega_{8}^5)^n),
\]
while for $r$ even we have:
\[
\#C_1(\F_{q^n}) = q^n + 1 -\frac{(q-1)(q-2\sqrt{q})}{4}(\sqrt{q})^n -\frac{(q-1)(q+2\sqrt{q})}{4}(-\sqrt{q})^n 
-\frac{(q-1)q}{4}((\sqrt{q} \, i)^n + (-\sqrt{q} \, i)^n).
\]
The characteristic polynomial of Frobenius is therefore
\[
P_1(X) = (X^2 -\sqrt{2q}X + q)^{\frac{(q-1)(q -\sqrt{2q})}{4}} (X^2+\sqrt{2q}X + q)^{\frac{(q-1)(q+\sqrt{2q})}{4}} 
\]
if $r$ is odd, and
\[ 
P_1(X) = (X-\sqrt{q})^{\frac{(q-1)(q - 2\sqrt{q})}{4}} (X +\sqrt{q})^{\frac{(q-1)(q + 2\sqrt{q})}{4}}  (X^2 + q)^{\frac{(q-1)q}{4}}
\]
if $r$ is even. These factors all arise from dimension one supersingular abelian varieties~\cite[Theorems 12.1 \& 12.2]{VJGaryAlexey}.



\subsection{$C_2/\F_q: y^{q}+y = x^{2q+1} + x^{q+2}$}
By Fourier-analysing the formulae in Table~\ref{table:C2points}, for $r$ odd we have:
\begin{eqnarray*}
\#C_2(\F_{q^n}) = q^n + 1 &-&\frac{(q-1)(q-\sqrt{2q})}{4}((\sqrt{q}\omega_8)^n + (\sqrt{q}\omega_{8}^7)^n)\\ 
                              &-&\frac{(q-1)(q + \sqrt{2q})}{4}((\sqrt{q}\omega_{8}^3)^n + (\sqrt{q}\omega_{8}^5)^n)\\
                              &-&\frac{(q-1)q}{2}((\sqrt{q} \, i)^n + (-\sqrt{q} \, i)^n).
\end{eqnarray*}
Let $\omega_{6} = (1 + \sqrt{3}i)/2$. Then for $r$ even we have:
\begin{eqnarray*}
\#C_2(\F_{q^n}) = q^n + 1 &-&\frac{(q-1)(q-2\sqrt{q})}{12}(\sqrt{q})^n\\ 
                              &-&\frac{(q-1)(q+2\sqrt{q})}{12}(-\sqrt{q})^n\\ 
                              &-&\frac{(q-1)(q - \sqrt{q})}{3}((\sqrt{q}\omega_{6})^n + (\sqrt{q}\omega_{6}^{5})^n)\\
                              &-&\frac{(q-1)(q + \sqrt{q})}{3}((\sqrt{q}\omega_{6}^2)^n + (\sqrt{q}\omega_{6}^{4})^n)\\
                              &-&\frac{(q-1)q}{4}((\sqrt{q} \, i)^n + (-\sqrt{q} \, i)^n).
\end{eqnarray*}
The characteristic polynomial of Frobenius is therefore
\[
P_2(X) = (X^2-\sqrt{2q}X + q)^{\frac{(q-1)(q-\sqrt{2q})}{4}} (X^2+\sqrt{2q}X + q)^{\frac{(q-1)(q+\sqrt{2q})}{4}} 
(X^2 + q)^{\frac{(q-1)q}{2}}
\]
if $r$ is odd, and
\begin{eqnarray*}
P_2(X)=
&& (X-\sqrt{q})^{\frac{(q-1)(q - 2\sqrt{q})}{12}}  (X +\sqrt{q})^{\frac{(q-1)(q + 2\sqrt{q})}{12}} \cdot \\
&& (X^2 - \sqrt{q}X + q)^{\frac{(q-1)(q - \sqrt{q})}{3}}  (X^2 + \sqrt{q}X + q)^{\frac{(q-1)(q + \sqrt{q})}{3}} 
(X^2 + q)^{\frac{(q-1)q}{4}}
\end{eqnarray*}
if $r$ is even. Again, these factors all arise from dimension one supersingular abelian varieties~\cite[Theorems 12.1 \& 12.2]{VJGaryAlexey}.

\subsection{$C_3/\F_q: y^{q}+y = x^{2q+1} + x^{q+2} + x^{q+1} + x^2$}
By Fourier-analysing the formulae in Table~\ref{table:C3points}, for $r$ odd we have:
\begin{eqnarray*}
\#C_3(\F_{q^n}) = q^n + 1 &-&\frac{(q-2)q}{8}((\sqrt{q})^n + (-\sqrt{q})^n) \\
                                    &-&\frac{(q-2)q}{8}((\sqrt{q} \, i)^n + (-\sqrt{q} \, i)^n)\\
                              &-&\frac{(q+1)(q - \sqrt{2q})}{12}( (\sqrt{q}\omega_{24})^n + (\sqrt{q}\omega_{24}^{7})^n + 
                               (\sqrt{q}\omega_{24}^{17})^n + (\sqrt{q}\omega_{24}^{23})^n)\\
                              &-&\frac{(q+1)(q + \sqrt{2q})}{12}( (\sqrt{q}\omega_{24}^5)^n + (\sqrt{q}\omega_{24}^{11})^n + 
                               (\sqrt{q}\omega_{24}^{13})^n + (\sqrt{q}\omega_{24}^{19})^n)\\
                              &-&\frac{(q-2)(5q - 4\sqrt{2q})}{24}( (\sqrt{q}\omega_{8})^n + (\sqrt{q}\omega_{8}^{7})^n)\\
                              &-&\frac{(q-2)(5q + 4\sqrt{2q})}{24}( (\sqrt{q}\omega_{8}^3)^n + (\sqrt{q}\omega_{8}^{5})^n)\\
\end{eqnarray*}
For $r$ even we have:
\begin{eqnarray*}
\#C_3(\F_{q^n}) = q^n + 1 &-&\frac{(q-2\sqrt{q})(5q + 2\sqrt{q} - 4)}{24}(\sqrt{q})^n 
                               -\frac{(q+2\sqrt{q})(5q - 2\sqrt{q} - 4)}{24}(-\sqrt{q})^n\\
                              &-&\frac{q(5q - 8)}{24}((\sqrt{q} \, i)^n + (-\sqrt{q} \, i)^n)\\
                              &-&\frac{q^2}{8}( (\sqrt{q}\omega_{8})^n + (\sqrt{q}\omega_{8}^{7})^n)\\
                               &-&\frac{q^2}{8}(  (\sqrt{q}\omega_{8}^{3})^n + (\sqrt{q}\omega_{8}^{5})^n)\\
                              &-&\frac{(q-1)q}{12}( (\sqrt{q}\omega_{12})^n + (\sqrt{q}\omega_{12}^{5})^n + 
                                (\sqrt{q}\omega_{12}^{7})^n + (\sqrt{q}\omega_{12}^{11})^n)\\
                              &-&\frac{(q-1)(q - 2\sqrt{q})}{12}( (\sqrt{q}\omega_{6})^n + (\sqrt{q}\omega_{6}^{5})^n) \\
                              &-&\frac{(q-1)(q + 2\sqrt{q})}{12}( (\sqrt{q}\omega_{6}^2)^n + (\sqrt{q}\omega_{6}^{4})^n)             
\end{eqnarray*}

The characteristic polynomial of Frobenius is therefore
\begin{eqnarray*}
P_3(X) = 
&& (X^2 - q)^{\frac{(q-2)q}{8}} (X^2 + q)^{\frac{(q-2)q}{8}} (X^2-\sqrt{2q}X + q)^{\frac{(q-2)(5q-4\sqrt{2q})}{24}} (X^2+\sqrt{2q}X + q)^{\frac{(q-2)(5q+4\sqrt{2q})}{24}} \\
&& (X^4 - \sqrt{2q}X^3 + qX^2 -q\sqrt{2q}X + q^2)^{\frac{(q+1)(q-\sqrt{2q})}{12}} \\
&& (X^4 + \sqrt{2q}X^3 + qX^2 + q\sqrt{2q}X + q^2)^{\frac{(q+1)(q+\sqrt{2q})}{12}}
\end{eqnarray*}
if $r$ is odd, and 
\begin{eqnarray*}
P_3 = && (X-\sqrt{q})^{\frac{(q-2\sqrt{q})(5q + 2\sqrt{q} - 4)}{24}} (X + \sqrt{q})^{\frac{(q+2\sqrt{q})(5q - 2\sqrt{q} - 4)}{24}} \\
&& (X^2 + q)^{\frac{q(5q-8)}{24}} (X^2 - \sqrt{q}X + q)^{\frac{q^2}{8}}  (X^2 + \sqrt{q}X + q)^{\frac{q^2}{8}} \\
&& (X^4 - qX^2 + q^2)^{\frac{(q-1)q}{12}}
(X^2 - \sqrt{q}X + q)^{\frac{(q-1)(q - 2\sqrt{q})}{12}}  (X^2 + \sqrt{q}X + q)^{\frac{(q-1)(q + 2\sqrt{q})}{12}}
\end{eqnarray*}
if $r$ is even.


These factors all arise from dimension one and two supersingular abelian varieties~\cite[Theorems 12.1 \& 12.2]{VJGaryAlexey}.
We remark that the three curves $C_1$, $C_2$, and $C_3$ have fairly large automorphism groups, and it may well be feasible to compute their 
characteristic polynomials of Frobenius by taking quotients and computing the decomposition of their Jacobian's via the Kani-Rosen 
theorem~\cite{kanirosen}, rather than use the techniques adopted here.


Finally, combining the expressions for $\#C_1(\F_{q^n})$, $\#C_2(\F_{q^n})$ and $\#C_3(\F_{q^n})$ as per Lemma~\ref{lem:fitzgeneral} and Eq.~(\ref{simple}) gives the following theorem.

\begin{theorem}
For $r$ odd we have:
\begin{eqnarray*}
F_q(n,0,0,0) = q^{n-3} - q^{n/2 - 3}\Big(&&\frac{q(q-1)(q-2)}{8}( 1^n + (-1)^n )\\
                           &+&\frac{q(q-1)(q+2)}{8}( i^n + (-i)^n )\\
                           &+&\frac{(q-1)(q+1)(q - \sqrt{2q})}{12}( \omega_{24}^n + \omega_{24}^{7n} +
                               \omega_{24}^{17n} + \omega_{24}^{23n})\\
                           &+&\frac{(q-1)(q+1)(q + \sqrt{2q})}{12}( \omega_{24}^{5n} + \omega_{24}^{11n} + 
                               \omega_{24}^{13n} + \omega_{24}^{19n})\\
                           &+&\frac{(q-1)(q - \sqrt{2q})(5q + \sqrt{2q} + 4)}{24}( \omega_{8}^n + \omega_{8}^{7n})\\
                           &+&\frac{(q-1)(q + \sqrt{2q})(5q - \sqrt{2q} + 4)}{24}( \omega_{8}^{3n} + \omega_{8}^{5n})
                         \Big).
\end{eqnarray*}
For $r$ even we have:
\begin{eqnarray*}
F_q(n,0,0,0) = q^{n-3} - q^{n/2 - 3}\Big(&&\frac{(q-1)(q - 2\sqrt{q})(5q + 2\sqrt{q} + 4)}{24} \, 1^n\\
                          &+&\frac{(q-1)(q + 2\sqrt{q})(5q - 2\sqrt{q} + 4)}{24}(-1)^n \\
                         &+&\frac{q(q-1)(5q+4)}{24}( i^n + (-i)^n )\\
                        &+&\frac{q^2(q-1)}{8}( \omega_{8}^n + \omega_{8}^{7n} + 
                                \omega_{8}^{3n} + \omega_{8}^{5n})\\
                         &+&\frac{q(q-1)^2}{12}( \omega_{12}^n + \omega_{12}^{5n} + 
                                \omega_{12}^{7n} + \omega_{12}^{11n})\\
                         &+&\frac{(q-1)(q - \sqrt{q})(q - \sqrt{q} + 2)}{12}( \omega_{6}^n + \omega_{6}^{5n}) \\
                         &+&\frac{(q-1)(q + \sqrt{q})(q + \sqrt{q} + 2)}{12}( \omega_{6}^{2n} + \omega_{6}^{4n}) \Big). 
\end{eqnarray*}
\end{theorem}

\section{Concluding remarks}\label{sec:conclusion}

By Fourier-analysing known formulae which count the number of elements of $\F_{2^n}$ for which the first three coefficients of the characteristic polynomial with respect to $\F_2$ are prescribed, we developed a new much simpler curve-based approach to deriving them, in the trace zero cases 
for all $n$ and the trace one cases for odd $n$. This approach was
used to count the number of irreducible polynomials in $\F_{2^r}[x]$ for which the first three coefficients are zero.
Based on the $\F_2$ base field case and our result for $f_q(n,0,0,0)$, we conjecture that for $q = 2^r$, $n \ge 3$ and $t_1,t_2,t_3 \in \F_q$, 
the formulae for $f_q(n,t_1,t_2,t_3)$ all have period $24$. 

For odd $n$ it is possible to compute all of the trace zero cases with our approach, but for reasons of space we only sketch the method here. 
First note that although one can express $N(0,0)$ in terms of $Z(q_1)$, $Z(q_2)$ and $Z(q_1 + q_2)$ as per Lemma~\ref{lem:fitzgeneral},
it is not possible to express all $N(t_2,t_3)$ in terms of these. For example, for any $t_2 \in \F_{q}^{\times}$ one has 
$Z(q_1) = Z(t_2 q_1) = \sum_{t_3 \in \F_q} N(0,t_3)$, so one does not get a linear system of full rank. However, this issue can be obviated
by changing the definition of $q_1$ and $q_2$ to:
\begin{eqnarray}
\nonumber q_1 &:& \F_{q^n} \rightarrow \F_q: x \mapsto T_1(x^{q+1} + x^2) + t_2,\\
\nonumber q_2 &:& \F_{q^n} \rightarrow \F_q: x \mapsto T_1(x^{2q+1} + x^{q+2}) + t_3.
\end{eqnarray}
Then, for odd $n$ one has $T_1(t_2) = t_2$ and $T_1(t_3) = t_3$, and so 
computing the number of points on the supersingular curves $C': y^q + y = x^{q+1} + x^2 + t_2$,
$C_{\alpha}': y^q + y = x^{2q+1} + x^{q+2} + t_3 + \alpha(x^{q+1} + x^2 + t_2)$ for $\alpha \in \F_q$ will
give expressions for $F_{q}(n,0,t_2,t_3)$.

\bibliographystyle{plain}
\bibliography{supersingular18thJuly}

\section*{Appendix: point and trace counts}

\begin{table}[h]
\caption{The number of $\F_{2^{rn}}$-rational points on $C_1:y^{2^r}+y=q_1(x)$}
\begin{center}\label{table:C1points}
\begin{tabular}{c|c|c}
\hline
$n \pmod{8}$ & $r$ odd &  $r$ even \\
\hline
$0$ & $(2^{rn}+1)-(2^r-1)2^{r(n+2)/2}$  &  $ (2^{rn}+1)-(2^r-1)2^{r(n+2)/2}$  \\
$1$ &  $(2^{rn}+1)+(2^r-1)2^{r(n+1)/2}$ &  $(2^{rn}+1)+(2^r-1)2^{r(n+1)/2}$ \\
$2$ & $2^{rn}+1$ & $2^{rn}+1$\\
$3$ & $(2^{rn}+1)-(2^r-1)2^{r(n+1)/2}$ &  $(2^{rn}+1)+(2^r-1)2^{r(n+1)/2}$ \\
$4$ & $(2^{rn}+1)+(2^r-1)2^{r(n+2)/2}$  &  $ (2^{rn}+1)-(2^r-1)2^{r(n+2)/2}$   \\	
$5$ &  $(2^{rn}+1)-(2^r-1)2^{r(n+1)/2}$ &  $(2^{rn}+1)+(2^r-1)2^{r(n+1)/2}$  \\
$6$ & $2^{rn}+1$   &  $2^{rn}+1$      \\
$7$ & $(2^{rn}+1)+(2^r-1)2^{r(n+1)/2}$ &  $(2^{rn}+1)+(2^r-1)2^{r(n+1)/2}$  \\
\hline
\end{tabular}
\end{center}
\end{table}

\begin{table}[h]
\caption{The number of $\F_{2^{rn}}$-rational points on $C_2:y^{2^r}+y=q_2(x)$}
\begin{center}\label{table:C2points}
\begin{tabular}{c|c|c}
\hline
$n \pmod{24}$ & $r$ odd &  $r$ even \\
\hline
$0$ &  $(2^{rn}+1)-(2^r-1)2^{r(n+2)/2+1}$  &  $(2^{rn}+1)-(2^r-1)2^{r(n+2)/2+1}$  \\
$1$ &  $(2^{rn}+1)+(2^r-1)2^{r(n+1)/2}$ &  $(2^{rn}+1)+(2^r-1)2^{r(n+1)/2}$ \\
$2$ & $(2^{rn}+1)+(2^r-1)2^{r(n+2)/2}$ & $(2^{rn}+1)+(2^r-1)2^{r(n+2)/2}$\\
$3$ & $(2^{rn}+1)-(2^r-1)2^{r(n+1)/2}$ &  $(2^{rn}+1)-(2^r-1)2^{r(n+1)/2}$ \\
$4$ &  $2^{rn}+1$  &  $2^{rn}+1$    \\	
$5$ &   $(2^{rn}+1)-(2^r-1)2^{r(n+1)/2}$ &  $(2^{rn}+1)+(2^r-1)2^{r(n+1)/2}$  \\
$6$ & $(2^{rn}+1)+(2^r-1)2^{r(n+2)/2}$   &  $(2^{rn}+1)-(2^r-1)2^{r(n+2)/2}$      \\
$7$ & $(2^{rn}+1)+(2^r-1)2^{r(n+1)/2}$ &  $(2^{rn}+1)+(2^r-1)2^{r(n+1)/2}$  \\
$8$ & $(2^{rn}+1)-(2^r-1)2^{r(n+2)/2+1}$  &  $2^{rn}+1$   \\
$9$ &  $(2^{rn}+1)+(2^r-1)2^{r(n+1)/2}$ &  $(2^{rn}+1)-(2^r-1)2^{r(n+1)/2}$  \\
$10$ &  $(2^{rn}+1)+(2^r-1)2^{r(n+2)/2}$   &  $(2^{rn}+1)+(2^r-1)2^{r(n+2)/2}$      \\
$11$ & $(2^{rn}+1)-(2^r-1)2^{r(n+1)/2}$ &  $(2^{rn}+1)+(2^r-1)2^{r(n+1)/2}$ \\
$12$ & $2^{rn}+1$  &  $2^{rn}+1-(2^r-1)2^{r(n+2)/2+1}$   \\
$13$ &   $(2^{rn}+1)-(2^r-1)2^{r(n+1)/2}$ &  $(2^{rn}+1)+(2^r-1)2^{r(n+1)/2}$  \\
$14$ &  $(2^{rn}+1)+(2^r-1)2^{r(n+2)/2}$   &  $(2^{rn}+1)+(2^r-1)2^{r(n+2)/2}$      \\
$15$ &  $(2^{rn}+1)+(2^r-1)2^{r(n+1)/2}$ &  $(2^{rn}+1)-(2^r-1)2^{r(n+1)/2}$ \\
$16$ &  $(2^{rn}+1)-(2^r-1)2^{r(n+2)/2+1}$  &  $2^{rn}+1$   \\
$17$ &   $(2^{rn}+1)+(2^r-1)2^{r(n+1)/2}$ &  $(2^{rn}+1)+(2^r-1)2^{r(n+1)/2}$ \\
$18$ &  $(2^{rn}+1)+(2^r-1)2^{r(n+2)/2}$   &  $(2^{rn}+1)-(2^r-1)2^{r(n+2)/2}$     \\
$19$ &  $(2^{rn}+1)-(2^r-1)2^{r(n+1)/2}$ &  $(2^{rn}+1)+(2^r-1)2^{r(n+1)/2}$  \\
$20$ &  $2^{rn}+1$  &  $2^{rn}+1$   \\
$21$ &  $(2^{rn}+1)-(2^r-1)2^{r(n+1)/2}$ &  $(2^{rn}+1)-(2^r-1)2^{r(n+1)/2}$\\
$22$ &  $(2^{rn}+1)+(2^r-1)2^{r(n+2)/2}$   &  $(2^{rn}+1)+(2^r-1)2^{r(n+2)/2}$     \\
$23$ &  $(2^{rn}+1)+(2^r-1)2^{r(n+1)/2}$ &  $(2^{rn}+1)+(2^r-1)2^{r(n+1)/2}$  \\
\hline
\end{tabular}
\end{center}
\end{table}

\begin{table}[h]
\caption{The number of $\F_{2^{rn}}$-rational points on $C_3:y^{2^r}+y=q_3(x)$ when $x^3+x+\alpha^{-1}$ has no root in $\mathbb F_{2^r}$}
\begin{center}\label{table:C3pointsnoroots}
\begin{tabular}{c|c|c}
\hline
$n \pmod{24}$ & $r$ odd &  $r$ even \\
\hline
$0$ & $(2^{rn}+1)-2^{r(n+2)/2+1}$  &  $(2^{rn}+1)-2^{r(n+2)/2+1}$  \\
$1$ & $(2^{rn}+1)+2^{r(n+1)/2}$ & 	$(2^{rn}+1)+2^{r(n+1)/2}$ \\
$2$ & $2^{rn}+1$ &  $2^{rn}+1$   \\
$3$ &    $(2^{rn}+1)+2^{r(n+1)/2+1}$ & 	$(2^{rn}+1)-2^{r(n+1)/2+1}$ \\
$4$ &   $(2^{rn}+1)-2^{r(n+2)/2}$ & 		$(2^{rn}+1)+2^{r(n+2)/2}$ \\	
$5$ &   $(2^{rn}+1)-2^{r(n+1)/2}$ & 		$(2^{rn}+1)+2^{r(n+1)/2}$ \\
$6$  &  $2^{rn}+1$   &  $2^{rn}+1$   \\
$7$ &  $(2^{rn}+1)+2^{r(n+1)/2}$ &  $(2^{rn}+1)+2^{r(n+1)/2}$  \\
$8$  &   $(2^{rn}+1)+2^{r(n+2)/2}$ &  $(2^{rn}+1)+2^{r(n+2)/2}$    \\
$9$ &   $(2^{rn}+1)-2^{r(n+1)/2+1}$  & $(2^{rn}+1)-2^{r(n+1)/2+1}$ \\
$10$  &  $2^{rn}+1$   &  $2^{rn}+1$    \\
$11$ &   $(2^{rn}+1)-2^{r(n+1)/2}$ &$(2^{rn}+1)+2^{r(n+1)/2}$  \\
$12$    &  $(2^{rn}+1)+2^{r(n+2)/2+1}$  & $(2^{rn}+1)-2^{r(n+2)/2+1}$  \\
$13$ &  $(2^{rn}+1)-2^{r(n+1)/2}$ &  $(2^{rn}+1)+2^{r(n+1)/2}$ \\
$14$  &  $2^{rn}+1$   &  $2^{rn}+1$     \\
$15$ &   $(2^{rn}+1)-2^{r(n+1)/2+1}$ & $(2^{rn}+1)-2^{r(n+1)/2+1}$   \\
$16$  &  $(2^{rn}+1)+2^{r(n+2)/2}$ &  $(2^{rn}+1)+2^{r(n+2)/2}$  \\
$17$ &  $(2^{rn}+1)+2^{r(n+1)/2}$ &  $(2^{rn}+1)+2^{r(n+1)/2}$  \\
$18$  &  $2^{rn}+1$   &  $2^{rn}+1$     \\
$19$ &  $(2^{rn}+1)-2^{r(n+1)/2}$ &  $(2^{rn}+1)+2^{r(n+1)/2}$   \\
$20$  &   $(2^{rn}+1)-2^{r(n+2)/2}$ & $(2^{rn}+1)+2^{r(n+2)/2}$    \\
$21$ & $(2^{rn}+1)+2^{r(n+1)/2+1}$ & $(2^{rn}+1)-2^{r(n+1)/2+1}$ \\
$22$  &  $2^{rn}+1$   &  $2^{rn}+1$    \\
$23$ &  $(2^{rn}+1)+2^{r(n+1)/2}$ &  $(2^{rn}+1)+2^{r(n+1)/2}$  \\
\hline
\end{tabular}
\end{center}
\end{table}

\begin{table}[h]
\caption{The number of $\F_{2^{rn}}$-rational points on $C_3:y^{2^r}+y=q_3(x)$}
\begin{center}\label{table:C3points}
\begin{tabular}{c|c|c}
\hline
$n \pmod{24}$ & $r$ odd &  $r$ even \\
\hline
$0$  &  $(2^{rn}+1)-(2^r-1)2^{r(n+2)/2+1}$  &  $(2^{rn}+1)-(2^r-1)2^{r(n+2)/2+1}$  \\
$1$ &  $(2^{rn}+1)+(2^r-1)2^{r(n+1)/2}$ &  $(2^{rn}+1)+(2^r-1)2^{r(n+1)/2}$ \\
$2$  &  $2^{rn}+1$ & $2^{rn}+1$\\
$3$ &  $(2^{rn}+1)+2^{r(n+1)/2+1}$ &  $2^{rn}+1$ \\
$4$  &  $(2^{rn}+1)-2^{r(n+2)/2}$  &  $(2^{rn}+1)+2^{r(n+2)/2}$    \\	
$5$ &  $(2^{rn}+1)-(2^r-1)2^{r(n+1)/2}$ &  $(2^{rn}+1)+(2^r-1)2^{r(n+1)/2}$  \\
$6$  &  $2^{rn}+1$   &  $2^{rn}+1$   \\
$7$ &  $(2^{rn}+1)+(2^r-1)2^{r(n+1)/2}$ &  $(2^{rn}+1)+(2^r-1)2^{r(n+1)/2}$  \\
$8$  &  $(2^{rn}+1)-(2^r-3)2^{r(n+2)/2}$  &  $(2^{rn}+1)-(2^r-1)2^{r(n+2)/2}$   \\
$9$ &  $(2^{rn}+1)-2^{r(n+1)/2+1}$ &  $2^{rn}+1$  \\
$10$  &  $2^{rn}+1$   &  $2^{rn}+1$      \\
$11$ &  $(2^{rn}+1)-(2^r-1)2^{r(n+1)/2}$ &  $(2^{rn}+1)+(2^r-1)2^{r(n+1)/2}$ \\
$12$  &  $(2^{rn}+1)+2^{r(n+4)/2}$  &  $2^{rn}+1-(2^r-2)2^{r(n+2)/2}$   \\
$13$ &  $(2^{rn}+1)-(2^r-1)2^{r(n+1)/2}$ &  $(2^{rn}+1)+(2^r-1)2^{r(n+1)/2}$  \\
$14$  &  $2^{rn}+1$   &  $2^{rn}+1$      \\
$15$ &  $(2^{rn}+1)-2^{r(n+1)/2+1}$ &  $2^{rn}+1$ \\
$16$  &  $(2^{rn}+1)-(2^r-3)2^{r(n+2)/2}$  &  $(2^{rn}+1)-(2^r-1)2^{r(n+2)/2}$   \\
$17$ &  $(2^{rn}+1)+(2^r-1)2^{r(n+1)/2}$ &  $(2^{rn}+1)+(2^r-1)2^{r(n+1)/2}$ \\
$18$  &  $2^{rn}+1$   &  $2^{rn}+1$     \\
$19$ &  $(2^{rn}+1)-(2^r-1)2^{r(n+1)/2}$ &  $(2^{rn}+1)+(2^r-1)2^{r(n+1)/2}$  \\
$20$  &  $(2^{rn}+1)-2^{r(n+2)/2}$  &  $(2^{rn}+1)+2^{r(n+2)/2}$   \\
$21$ &$(2^{rn}+1)+2^{r(n+1)/2+1}$ &  $2^{rn}+1$\\
$22$  &  $2^{rn}+1$   &  $2^{rn}+1$     \\
$23$ &  $(2^{rn}+1)+(2^r-1)2^{r(n+1)/2}$ &  $(2^{rn}+1)+(2^r-1)2^{r(n+1)/2}$  \\
\hline
\end{tabular}
\end{center}
\end{table}

\begin{table}[h]
\caption{$F_q(n,0,0,0)$}
\begin{center}\label{table:Fqn000}
\begin{tabular}{c|c|c}
\hline
$n \pmod{24}$ &  $r$ odd &  $r$ even \\
\hline
$0$  &  $q^{n-3} -(q-1)(2q+1)q^{n/2-2}$   &  $q^{n-3} -(q-1)(2q+1)q^{n/2-2}$ \\
$1$ &  $q^{n-3} + (q^{2}-1)q^{(n-1)/2-2}$ &  $q^{n-3} + (q^{2}-1)q^{(n-1)/2-2}$\\
$2$  &  $q^{n-3} + (q-1)q^{n/2-2}$ &  $q^{n-3} + (q-1)q^{n/2-2}$\\
$3$ &  $q^{n-3}$  &  $q^{n-3}$  \\
$4$  &  $q^{n-3}$   &  $q^{n-3}$  \\	
$5$ &  $q^{n-3} -(q^{2}-1)q^{(n-1)/2-2}$  &  $q^{n-3} + (q^{2}-1)q^{(n-1)/2-2}$ \\
$6$  & $q^{n-3} + (q-1)q^{n/2-2}$  & $q^{n-3} -(q-1)q^{n/2-2}$ \\
$7$ &  $q^{n-3} + (q^{2}-1)q^{(n-1)/2-2}$  &  $q^{n-3} + (q^{2}-1)q^{(n-1)/2-2}$ \\
$8$  &  $q^{n-3}-(q-1)q^{n/2-1}$  &  $q^{n-3}-(q-1)q^{n/2-1}$  \\
$9$ &  $q^{n-3}$  &  $q^{n-3}$  \\
$10$  &  $q^{n-3} + (q-1)q^{n/2-2}$ &  $q^{n-3} + (q-1)q^{n/2-2}$  \\
$11$ &  $q^{n-3} -(q^{2}-1)q^{(n-1)/2-2}$ &  $q^{n-3} + (q^{2}-1)q^{(n-1)/2-2}$ \\
$12$  &  $q^{n-3} + (q^{2}-1)q^{n/2-2}$  &  $q^{n-3} - (q^{2}-1)q^{n/2-2}$   \\
$13$ &  $q^{n-3} - (q^{2}-1)q^{(n-1)/2-2}$ &  $q^{n-3} + (q^{2}-1)q^{(n-1)/2-2}$ \\
$14$  &  $q^{n-3} + (q-1)q^{n/2-2}$   &  $q^{n-3} + (q-1)q^{n/2-2}$    \\
$15$ &  $q^{n-3}$ &  $q^{n-3}$ \\
$16$  &  $q^{n-3}-(q-1)q^{n/2-1}$ &  $q^{n-3}-(q-1)q^{n/2-1}$    \\
$17$ &  $q^{n-3} + (q^{2}-1)q^{(n-1)/2-2}$  &  $q^{n-3} + (q^{2}-1)q^{(n-1)/2-2}$ \\
$18$  & $q^{n-3} + (q-1)q^{n/2-2}$  & $q^{n-3} - (q-1)q^{n/2-2}$  \\
$19$ &  $q^{n-3} - (q^{2}-1)q^{(n-1)/2-2}$ &  $q^{n-3} + (q^{2}-1)q^{(n-1)/2-2}$   \\
$20$  &  $q^{n-3}$  &  $q^{n-3}$  \\
$21$ &  $q^{n-3}$  &  $q^{n-3}$\\
$22$  &  $q^{n-3} + (q-1)q^{n/2-2}$  &  $q^{n-3} + (q-1)q^{n/2-2}$     \\
$23$ &  $q^{n-3} + (q^{2}-1)q^{(n-1)/2-2}$ &  $q^{n-3} + (q^{2}-1)q^{(n-1)/2-2}$ \\
\hline
\end{tabular}
\end{center}
\end{table}

\end{document}